\nonstopmode
\documentclass[10pt,reqno]{amsart}
\usepackage{graphicx}
\usepackage{latexsym}
\usepackage{fancyhdr}
\usepackage{amsmath, amssymb,amsthm}
\usepackage[all]{xy}
\usepackage{pdflscape}
\usepackage{longtable}
\usepackage{rotating}
\usepackage{verbatim}
\usepackage{hyperref}
\usepackage{cleveref}
\usepackage{subfigure}
\usepackage{mathrsfs}
\usepackage{mdwlist}
\usepackage{dsfont}
\usepackage{mathtools}
\usepackage{float}
\usepackage{color}
\usepackage{stmaryrd}

\usepackage{enumitem}
\makeatletter
\newcommand{\mylabel}[2]{#2\def\@currentlabel{#2}\label{#1}}
\makeatother
\usepackage{tikz}
\usepackage{tikz-cd}

\definecolor{teal}{rgb}{0.0, 0.5, 0.5}

\newcounter{mnotecount}[section]

\newcommand{\rmnote}[1]{}%{\mnote{#1}}

\allowdisplaybreaks

\DeclareFontFamily{U}{mathb}{\hyphenchar\font45}
\DeclareFontShape{U}{mathb}{m}{n}{
	<5> <6> <7> <8> <9> <10> gen * mathb
	<10.95> mathb10 <12> <14.4> <17.28> <20.74> <24.88> mathb12
}{}
\DeclareSymbolFont{mathb}{U}{mathb}{m}{n}
\DeclareFontSubstitution{U}{mathb}{m}{n}

\let\dot\relax
\DeclareMathAccent{\dot}{0}{mathb}{"39}
\let\ddot\relax
\DeclareMathAccent{\ddot}{0}{mathb}{"3A}
\let\dddot\relax
\DeclareMathAccent{\dddot}{0}{mathb}{"3B}
\let\ddddot\relax
\DeclareMathAccent{\ddddot}{0}{mathb}{"3C}

\theoremstyle{plain}
\newtheorem*{theorem*}{Theorem}
\newtheorem{theorem}{Theorem}[section]
\newtheorem*{lemma*}{Lemma}
\newtheorem{lemma}[theorem]{Lemma}
\newtheorem*{assumption*}{Assumption}

\newtheorem*{proposition*}{Proposition}
\newtheorem{proposition}[theorem]{Proposition}
\newtheorem*{corollary*}{Corollary}
\newtheorem{corollary}[theorem]{Corollary}
\newtheorem*{claim*}{Claim}

\newtheorem*{step*}{Step}

\newtheorem*{conjecture*}{Conjecture}

\newtheorem*{question*}{Question}
\newtheorem*{result*}{Result}

\newtheorem{supplement}[theorem]{Supplement}
\theoremstyle{definition}
\newtheorem*{definition*}{Definition}

\newtheorem*{example*}{Example}

\newtheorem*{algorithm*}{Algorithm}
\newtheorem*{remark*}{Remark}
\newtheorem*{remarks*}{Remarks}
\newtheorem{remark}[theorem]{Remark}

\newtheorem*{convention*}{Convention}

\numberwithin{equation}{section}

\def\al{\alpha}
\def\be{\beta}
\def\ga{\gamma}
\def\de{\delta}
\def\ep{\epsilon}
\def\ve{\varepsilon}
\renewcommand{\ep}{\ve}

\def\et{\eta}
\def\th{\theta}
\def\vt{\vartheta}
\def\ka{\kappa}
\def\la{\lambda}

\def\rh{\rho}

\def\si{\sigma}

\def\ph{\phi}
\def\vh{\varphi}
\renewcommand{\ph}{\vh}
\def\ch{\chi}

\def\om{\omega}

\def\Ga{\Gamma}

\def\La{\Lambda}

\def\Om{\Omega}

\def\C{\mathbb{C}}

\def\N{\mathbb{N}}

\def\R{\mathbb{R}}

\def\cA{\mathcal{A}}

\def\cD{\mathcal{D}}
\def\cE{\mathcal{E}}
\def\cF{\mathcal{F}}
\def\cG{\mathcal{G}}
\def\cH{\mathcal{H}}

\def\cP{\mathcal{P}}

\def\fE{\mathfrak{E}}

\def\fM{\mathfrak{M}}
\def\fN{\mathfrak{N}}

\def\db{\ol \partial}

\renewcommand{\Im}{\on{Im}}
\def\<{\langle}
\def\>{\rangle}
\renewcommand{\o}{\circ}

\renewcommand{\preceq}{\preccurlyeq}

\def\ol{\overline}
\def\ul{\underline}

\let\on=\operatorname

\newcommand{\sr}[1]%
{\ifmmode{}^\dagger\else${}^\dagger$\fi\ifvmode
	\vbox to 0pt{\vss
		\hbox to 0pt{\hskip\hsize\hskip1em
			\vbox{\hsize3cm\raggedright\pretolerance10000
				\noindent #1\hfill}\hss}\vss}\else
	\vadjust{\vbox to0pt{\vss%
			\hbox to 0pt{\hskip\hsize\hskip1em%
				\vbox{\hsize3cm\raggedright\pretolerance10000%
					\noindent #1\hfill}\hss}\vss}}\fi%
}

\def\A{\;\forall}
\def\E{\;\exists}

\providecommand{\mapsfrom}{\kern.2em%
	\setbox0=\hbox{$\leftarrow$\kern-.10em\rule[0.26mm]{0.1mm}{1.3mm}}\box0%
	\kern.3em}

\title[The Borel map in the mixed Beurling setting]
{The Borel map in the mixed Beurling setting}

\author[D.N.~Nenning, A.~Rainer, and G.~Schindl]{David Nicolas Nenning, Armin Rainer, and Gerhard Schindl}

\address{Fakult\"at f\"ur Mathematik, Universit\"at Wien, Oskar-Morgenstern-Platz~1, A-1090 Wien, Austria.}
\email{david.nicolas.nenning@univie.ac.at}
\email{armin.rainer@univie.ac.at}
\email{gerhard.schindl@univie.ac.at}

\begin{document}

	\begin{abstract}
		The Borel map takes a smooth function to its infinite jet of derivatives (at zero).
		We study the restriction of this map to ultradifferentiable classes of Beurling type
		in a very general setting which encompasses the classical Denjoy--Carleman and
		Braun--Meise--Taylor classes.
		More precisely, we characterize when the Borel image of one class covers the sequence space of another class
		in terms of the two weights that define the classes.
		We present two independent solutions to this problem, one by reduction to the Roumieu case and
		the other by dualization of the involved Fr\'echet spaces, a Phragm\'en--Lindel\"of theorem, and H\"ormander's
		solution of the $\db$-problem.
	\end{abstract}

	\thanks{AR was supported by FWF-Project P 32905-N, DNN and GS by FWF-Project P 33417-N}
	\date{\today}

	\maketitle

	%\tableofcontents

	\section{Introduction}

	The Borel map $j^\infty_0 : C^\infty(\R) \to \C^\N$ at $0$ is defined by $j^\infty_0 f := (f^{(n)}(0))_{n \in \N}$.
	We will study the restriction of $j^\infty_0$ to ultradifferentiable classes in a general setting
	which allows us to treat the classical Denjoy--Carleman and Braun--Meise--Taylor classes
	at the same time. Our classes are defined in terms of one-parameter families $\fM = (M^{(x)})_{x>0}$,
	$\fN = (N^{(x)})_{x>0}$, etc.,
	of weight sequences; we call them weight matrices.
	In this article,
	we are mostly interested in classes of Beurling type
	\[
	\cE^{(\fN)}(\R) := \Big\{f \in C^\infty(\R) : \A j,k,l \in \N_{\ge 1}:
	\sup_{x \in [-j,j]}\sup_{p \in \N} \frac{|f^{(p)}(x)|}{(\frac{1}{l})^p
		N^{(\frac{1}{k})}_p} < \infty  \Big\},
	\]
	but we shall have to use also results on the Roumieu counterpart
	\[
	\cE^{\{\fN\}}(\R) := \Big\{f \in C^\infty(\R) : \A j \in \N_{\ge 1}  \E k,l \in \N_{\ge 1}:
	\sup_{x \in [-j,j]}\sup_{p \in \N} \frac{|f^{(p)}(x)|}{l^p N^{(k)}_p} < \infty  \Big\}.
	\]
	By definition, the image $j^\infty_0 \cE^{(\fN)}(\R)$ is contained in the sequence space
	\[
	\La^{(\fN)} := \Big\{a=(a_p) \in \C^\N : \A k,l \in \N_{\ge 1}:
	\sup_{p \in \N} \frac{|a_p|}{(\frac{1}{l})^p
		N^{(\frac{1}{k})}_p} < \infty \Big\};
	\]
	and likewise $j^\infty_0 \cE^{\{\fN\}}(\R) \subseteq \La^{\{\fN\}}$, where $\La^{\{\fN\}}$ is defined analogously.
	The goal of this paper is to
	find necessary and sufficient conditions for
	\begin{equation}
		\label{eq:BorelImage}
		\Lambda^{(\fM)} \subseteq j^\infty_0 \cE^{(\fN)}(\R),
	\end{equation}
	in terms of $\fM$ and $\fN$.

	The Roumieu case is well understood, see our recent article \cite{RoumieuMoment}:
	under some mild assumptions on $\fM$ and $\fN$, we have
	$\Lambda^{\{\fM\}} \subseteq j^\infty_0 \cE^{\{\fN\}}(\R)$ if and only if
	\begin{equation*}
		\A x>0 \E y>0 : ~ M^{(x)} \prec_{SV} N^{(y)},
	\end{equation*}
	where $M^{(x)} \prec_{SV} N^{(y)}$ means
	\begin{equation*}
		\E s \in \N:~ \sup_{j\ge 1}\sup_{0\le i < j}
		\Big( \frac{M^{(x)}_j}{s^j N^{(y)}_i} \Big)^{\frac{1}{j-i}}\frac{1}{j}\sum_{k = j}^\infty \frac{N^{(y)}_{k-1}}{N^{(y)}_k}<\infty,
	\end{equation*}
	a condition introduced by Schmets and Valdivia in \cite{surjectivity}.

	The characterization of \eqref{eq:BorelImage} is considerably more difficult (partly,
	because the image of an intersection is, in general, smaller than the intersection of the images).
	We solve this problem in two independent ways:
	\begin{enumerate}[label={\thetag{\arabic*}}]
		\item The first method reduces the Beurling to the Roumieu problem, and uses the solution of the latter.
		This is a well-known approach which has been used, in various disguises, in several settings; see e.g.\
		\cite{ChaumatChollet94}, \cite{surjectivity}, \cite{mixedramisurj}, and \cite{Rainer:2020ab}.
		The additional parameter (i.e., $x$ in the weight matrix) makes this delicate reduction quite involved.
		As a result we prove in \Cref{thm:mainthm1} that (again under mild assumptions) \eqref{eq:BorelImage} is equivalent to
		\begin{equation*}
			\A y>0 \E x>0 : ~ M^{(x)} \prec_{SV} N^{(y)}.
		\end{equation*}
		\item The second approach is based on dualization of \eqref{eq:BorelImage} and identification of the
		strong duals $(\La^{(\fM)})'$ and $\cE^{(\fN)}(\R)'$ with suitable spaces of entire functions.
		This strategy has been implemented by \cite{BonetMeiseTaylorSurjectivity} (for Braun--Meise--Taylor classes,
		following \cite{Carleson61} and \cite{Ehrenpreis70}).
		In fact, our analysis is based on the abstract functional-analytic result \cite[Corollary 2.3]{BonetMeiseTaylorSurjectivity}
		(which we restate in \Cref{prop:inclusion}).
		It translates the problem to a question about bounded sets in the mentioned spaces of entire functions,
		where a Phragm\'en--Lindel\"of theorem and H\"ormander's solution of the $\db$-problem
		can be brought to bear.
		We find in \Cref{thm:mainthm2} that (under other mild assumptions) \eqref{eq:BorelImage} is equivalent to
		\begin{equation*}
			\A y>0 \E x>0 : ~ M^{(x)} \prec_{L} N^{(y)}.
		\end{equation*}
	\end{enumerate}
	The condition $M^{(x)} \prec_{L} N^{(y)}$ means
	\begin{equation*}
		\E C>0 \A s\ge 0:~ \frac{s}{\pi} \int_{-\infty}^\infty \frac{\om_{N^{(y)}}(t)}{t^2+s^2}\,dt\le\om_{M^{(x)}}(Cs)+C,
	\end{equation*}
	where $\om_{M}(t) := \sup_{k \in \N} \log (\frac{t^k}{M_k})$ is the pre-weight function associated with a weight sequence $M$.
	It appears as (2.14') in Langenbruch's paper \cite{Langenbruch94} and
	it is closely related to the condition appearing in \cite{BonetMeiseTaylorSurjectivity},
	but with a little twist; see \Cref{constantsinout} and \Cref{sec:comparison}.

	Let us briefly describe the structure of the paper.
	In \Cref{sec:classes}, we gather all relevant notation
	and conditions concerning weight sequences, functions, and matrices
	and we introduce the corresponding ultradifferentiable function and sequence spaces.
	The solution by reduction (1) is obtained in \Cref{sec:reduction}.
	In \Cref{sec:duals}, we identify the duals $(\La^{(\fM)})'$ and $\cE^{(\fN)}(\R)'$ with
	certain weighted spaces of entire functions.
	This allows us to carry out the solution by dualization (2) in \Cref{sec:dualization}.
	In the final \Cref{sec:comparison}, we show that our theorems specialize to the known results for
	Denjoy--Carleman and Braun--Meise--Taylor classes; see \Cref{sequencesrleations}, \Cref{supplement},
	and \Cref{weighfunction2}.
	In the short appendix, we prove a technical statement needed in \Cref{sec:duals},
	namely that the entire functions are dense in an auxiliary function space.
	Since the inclusion of the entire functions is continuous, also the polynomials are dense.

	\section{Ultradifferentiable classes and weights} \label{sec:classes}

	Ultradifferentiable classes are weighted classes of smooth functions.

	\subsection{Weight sequences} \label{sec:sequence}

	We call a sequence of positive real numbers $M=(M_k)$ a \emph{weight sequence}, if $M_0 = 1$ and $M_k= \mu_1 \cdots \mu_k$, $k\ge 1$, for
	an increasing sequence $0<\mu_1 \le \mu_2 \le \cdots$ tending to $\infty$.
	We call a weight sequence \emph{normalized} if $\mu_1 \ge 1$ and put $\mu_0:=1$. Let us also set $m_k:=\frac{M_k}{k!}$.

	That $\mu_k$ is increasing means that $M_k$ is log-convex.
	Here are some easy consequences of the definition: $M_j M_k \le M_{j+k}$,
	$(M_k)^{1/k} \le \mu_k$, and $(M_k)^{1/k} \to \infty$ if and only if $\mu_k \to \infty$
	(cf.\ \cite[Lemma 2.3]{Rainer:2021aa}).

	For a weight sequence $M$, we define
	the \emph{Denjoy--Carleman  class of Beurling type}
	\[
	\cE^{(M)}(\R):= \Big\{ f \in C^\infty(\R) : \A K \subset\subset \R~\A r>0: \|f\|^M_{K,r}:= \sup_{x\in K, k\in \N}\frac{|f^{(k)}(x)|}{r^k M_k}<\infty \Big\}.
	\]
	It is endowed with the natural projective topology and thus has the structure of a Fr\'echet space.
	If the universal quantifier in front of $r$ is replaced by an existential quantifier one gets the Denjoy--Carleman class $\cE^{\{M\}}(\R)$ of \emph{Roumieu type}.

	It is immediate that the restriction of the Borel map $j^\infty_0$ to $\cE^{(M)}(\R)$ takes values in
	the corresponding sequence space
	\[
	\La^{(M)}:= \Big\{ \la=(\la_k)_k\in \C^{\N} : \A r>0: \|\la\|^M_r := \sup_{k\in \N}\frac{|\la_k|}{r^k M_k}<\infty \Big\},
	\]
	which again is endowed with its natural Fr\'echet topology.
	By the Denjoy-Carleman theorem, $j^\infty_0|_{\cE^{(M)}(\R)}$ is injective if and only if
	\[
	\sum_{k \ge 1} \frac{1}{\mu_k} = \infty;
	\]
	see e.g.\ \cite[Theorem 4.2]{Komatsu73}, \cite[Theorem 1.3.8]{Hoermander83I}, or \cite[Theorem 3.6]{Rainer:2021aa}.
	In that case, the class (and the weight sequence) is called \emph{quasianalytic}, and \emph{non-quasianalytic} otherwise.

	We say that $M$ has \emph{moderate growth}, if
	\[
	\E C >0~ \A j,k \in \N:\quad M_{j+k}\le C^{j+k}M_jM_k,
	\]
	and $M$ is \emph{derivation closedness}, if
	\[
	\E C >0~ \A j \in \N: \quad M_{j+1}\le C^{j+1}M_j.
	\]
	All these conditions are frequently used in the theory of ultradifferentiable classes;
	in \cite{Komatsu73}, non-quasianalyticity is denoted by $(M.3)'$, derivation closedness by $(M.2)'$ and moderate growth by $(M.2)$.

	Given two weight sequences $M$ and $N$, we write $M\le N$ if $M_k\le N_k$ for all $k$, and $M\preccurlyeq N$ if $\sup_{k>0}\big(\frac{M_k}{N_k}\big)^{1/k}<\infty$.
	We say that $M$ and $N$ are \emph{equivalent}
	if $M\preccurlyeq N$ and $N\preccurlyeq M$. Note that both moderate growth and derivation closedness are preserved under equivalence.
	Two weight sequences are equivalent if and only if they generate the same class. In fact,
	\[
	\cE^{(M)}(\R) \subseteq \cE^{(N)}(\R) \quad\Longleftrightarrow\quad    M\preccurlyeq N \quad\Longleftrightarrow\quad  \La^{(M)} \subseteq \La^{(N)}.
	\]
	We shall also need the relation $M \lhd N$, defined by $\lim_{k\to \infty}\big(\frac{M_k}{N_k}\big)^{1/k}=0$,
	which is equivalent to $\cE^{\{M\}}(\R) \subseteq \cE^{(N)}(\R)$ as well as $\La^{\{M\}} \subseteq \La^{(N)}$. All this can be found in \cite[Proposition 2.12]{compositionpaper}.

	\subsection{Weight functions}
	\label{sec:function}
	The second approach to ultradifferentiable classes is based on weight functions,
	i.e., increasing continuous functions $\om:[0,\infty) \rightarrow [0,\infty)$ satisfying some additional properties
	which will be specified shortly.
	Originally, $\om$ was used, by Beurling \cite{Beurling61} and Bj\"orck \cite{Bjorck66},
	to impose growth restrictions at infinity on the Fourier transform of the functions in question.
	In the modern approach due to Braun, Meise, and Taylor \cite{BraunMeiseTaylor90}, the derivatives of the functions
	are controlled
	by the \emph{Young conjugate} of $y\mapsto \ph_\om(y):=\om(e^y)$, that is
	\[
	\ph_\om^*(x):= \sup\{xy - \ph_\om(y):~y \ge 0\},\;\;\;x\ge 0.
	\]
	Assuming that $\log(t) = o(\om(t))$ as $t \rightarrow \infty$, which ensures that $\ph_\om^*(x)$ is finite for all $x>0$, one defines the
	\emph{Braun--Meise--Taylor class of Beurling type}
	\[
	\cE^{(\om)}(\R):= \Big\{ f \in C^\infty(\R) : \A K \subset\subset \R \A r>0 : \|f\|^\om_{K,r}:= \sup_{x\in K, k\in \N}\frac{|f^{(k)}(x)|}{e^{\ph_\om^*(rk)/r}}<\infty \Big\}
	\]
	and
	endows it with the natural Fr\'echet topology.
	Similarly, we have the Fr\'echet space
	\[
	\La^{(\om)}:= \Big\{\la=(\la_k)_k\in \C^\N : \A r>0: \|\la\|^\om_{r}:= \sup_{k\in \N}\frac{|\la_k|}{e^{\ph_\om^*(rk)/r}}<\infty \Big\}
	\]
	and the map $j^\infty_0|_{\cE^{(\om)}(\R)} : \cE^{(\om)}(\R) \to \La^{(\om)}$.
	Again there is a Roumieu version of these classes of functions and sequences, where $r$ is subjected to an existential quantifier; we refer to \cite{RoumieuMoment} for details.

	Let us now make precise the relevant regularity properties for $\om$.
	We say that an increasing continuous function $\om:[0,\infty)\rightarrow [0,\infty)$ with $\omega(0)=0$ is a \emph{pre-weight function}, if
	$\log(t)=o(\omega(t))$ as $t\rightarrow\infty$ (in particular, $\omega(t) \to \infty$),
	and $\varphi_{\omega}$ is convex.
	We call a pre-weight $\omega$ a \emph{weight function} if it also fulfills
	\begin{equation}
		\label{om1}
		\tag{$\om_1$}
		\omega(2t)=O(\omega(t)) \text{ as } t\rightarrow\infty.
	\end{equation}
	The map $j^\infty_0|_{\cE^{(\om)}(\R)}$ is injective if and only if
	\[
	\int_0^\infty \frac{\om(t)}{1+t^2}\, dt =\infty;
	\]
	see e.g.\ \cite{BraunMeiseTaylor90}, \cite[Section 4]{testfunctioncharacterization}, or \cite[Theorem 11.17]{Rainer:2021aa}.
	Then the class and the weight function $\om$ are called \emph{quasianalytic}, and \emph{non-quasianalytic} otherwise.
	It is straightforward to see that non-quasianalyticity implies $\om(t)=o(t)$ as $t\rightarrow\infty$.

	Two pre-weight functions are called \emph{equivalent},
	written $\om\sim\si$, if $\om(t)=O(\si(t))$ and $\si(t)=O(\om(t))$ as $t\to \infty$. This is precisely the case if they generate the same classes.
	Indeed,
	\[
	\cE^{(\om)}(\R) \subseteq \cE^{(\si)}(\R) \Longleftrightarrow  \La^{(\om)} \subseteq \La^{(\si)} \Longleftrightarrow  \si(t)=O(\om(t)) \text{ as } t \to \infty,
	\]
	see \cite[Corollary 5.17]{compositionpaper}. For every pre-weight function there is an equivalent pre-weight function which vanishes on $[0,1]$.

	\begin{remark}
		We will frequently consider the radially symmetric extension $\C \ni z \mapsto \om(|z|)$ of a pre-weight function $\om$.
		By abuse of notation, we will still write $\om(z)$ instead of $\om(|z|)$.
	\end{remark}

	\subsection{The associated weight function}\label{assofctsection}

	Let $M$ be a weight sequence. Then
	\[
	\om_M(t):= \sup_{k \in \N} \log\Big(\frac{t^k}{M_k}\Big),
	\]
	is a pre-weight function;
	cf.\ \cite[Chapitre I]{mandelbrojtbook} and \cite[Section 3.1]{Komatsu73}.
	See \cite[Theorem 3.1]{subaddlike} for necessary and sufficient conditions for $\om_M$ being a weight function. For $\la >0$, we set $\mu_M(\la):=|\{p \in \N_{\ge 1}:~\mu_p\le \la \}|$. Then we have the following integral representation of $\om_M$, cf. e.g. \cite[(3.11)]{Komatsu73} and references therein,
	\begin{equation}
		\label{eq:countingrepr}
		\om_M(t)=\int_0^t \frac{\mu_M(\la)}{\la}\,d\la.
	\end{equation}

	If $M$ is normalized, then $\om_M|_{[0,1]}=0$.
	And $\om_M$ is non-quasianalytic if and only if $M$ is non-quasianalytic; see \cite[Lemma 4.1]{Komatsu73}.
	Note that a weight sequence $M$ can be recovered from $\om_M$ by
	\begin{equation}\label{assoweight}
		M_k=\sup_{t>0}\frac{t^k}{\exp(\om_M(t))},\quad k\in\N.
	\end{equation}
	In general, $\cE^{(M)}(\R)$ and $\cE^{(\om_M)}(\R)$ may differ, unless $M$ has moderate growth; see \cite{BonetMeiseMelikhov07} and \cite[Section 5]{compositionpaper}.

	\subsection{Weight matrices}
	\label{sec:matrices}

	In \cite{compositionpaper} and \cite{dissertation}, Denjoy--Carleman  and Braun--Meise--Taylor classes
	were understood as special cases of ultradifferentiable classes defined by weight matrices.
	A \emph{weight matrix} is a one-parameter family of weight sequences $\fM=(M^{(x)})_{x>0}$
	such that $M^{(x)}\le M^{(y)}$ if $x\le y$ and, for all $x>0$,
	\begin{equation} \label{eq:realanalytic}
		(m^{(x)}_j)^{1/j} \to \infty \text{ as }j \to \infty.
	\end{equation}
	We call $\fM$ \emph{normalized} if all $M^{(x)}\in \fM$ are normalized.

	We define the classes of \emph{Beurling type}
	\[
	\cE^{(\fM)}(\R):= \Big\{ f \in C^\infty(\R) : \A K \subset\subset \R~\A r,x>0:~\|f\|_{K,r}^{M^{(x)}} <\infty \Big\},
	\]
	and
	\[
	\La^{(\fM)}:= \Big\{ \la=(\la_k)_k\in \C^{\N} : \A r,x>0:~\|\la\|_{r}^{M^{(x)}} <\infty \Big\},
	\]
	and endow both spaces with their natural Fr\'echet topology.
	Note that, by our assumption \eqref{eq:realanalytic}, each class
	$\cE^{(\fM)}(\R)$ contains all real analytic functions on $\R$ (cf.\ \cite[Section 4.1]{compositionpaper}).

	If \emph{all} $M^{(x)}\in \fM$ are non-quasianalytic, we call $\fM$ \emph{non-quasianalytic}.
	Non-quasianalyticity of $\fM$ is equivalent to the existence of bump functions in $\cE^{(\fM)}(\R)$; see \cite[Proposition  4.7]{testfunctioncharacterization}
	or \cite[Theorem 11.16]{Rainer:2021aa}.

	Any weight sequence $M$ induces a weight matrix $\fM=(M^{(x)})_{x>0}$ with $M^{(x)}=M$ for all $x>0$.
	Then, obviously, $\cE^{(M)}(\R) = \cE^{(\fM)}(\R)$ and $\La^{(M)}(\R) = \La^{(\fM)}(\R)$.

	\subsection{Weight matrices associated with pre-weight functions}

	To a pre-weight function $\om$ (vanishing on $[0,1]$) such that $\om(t) = o(t)$ as $t \to \infty$,
	we assign the normalized weight matrix $\Om=(W^{(x)})_{x>0}$ defined by
	\begin{equation}\label{assosequences}
		W^{(x)}_k:=\exp\Big(\frac{1}{x} \ph_\om^*(xk)\Big).
	\end{equation}
	If $\om$ is actually a weight function, then
	\begin{equation}\label{omOMlocallyequal}
		\cE^{(\om)}(\R) \cong \cE^{(\Om)}(\R) \quad \text{ and } \quad \La^{(\om)} \cong \La^{(\Om)}
	\end{equation}
	as locally convex spaces; see \cite{compositionpaper} and \cite{dissertation}.
	Let us remark that here $\om(t) = o(t)$ as $t \to \infty$ is assumed so that $\Om$ satisfies our standard assumption \eqref{eq:realanalytic}.

	Let us collect some useful properties of $\Om$.

	\begin{lemma}\label{Ompropcollect}
		The weight matrix $\Om=(W^{(x)})_{x>0}$ satisfies:
		\begin{enumerate}[label=\thetag{\arabic*}]
			\item $\vartheta^{(x)}\le\vartheta^{(y)}$ if $x\le y$, where $\vt^{(x)}_k:= \frac{W^{(x)}_k}{W^{(x)}_{k-1}}$.
			\item $W^{(x)}_{j+k}\le W^{(2x)}_jW^{(2x)}_k$ for all $x>0$ and $j,k \in \N$. \label{newmoderategrowth}
			\item $\omega\sim\omega_{W^{(x)}}$ for each $x>0$. More precisely,
			\begin{equation}\label{goodequivalenceclassic}
				\A x>0 \E D_x>0  : x \om_{W^{(x)}} \le \om \le 2x\om_{W^{(x)}}+D_x.
			\end{equation}
			\item $(w^{(x)}_k)^{1/k} \to \infty$ for all $x>0$ if and only if $\om(t)=o(t)$ as $t \to \infty$.
			\item $\om$ is non-quasianalytic if and only if each $W^{(x)}$ is non-quasianalytic, i.e., if and only if $\Om$ is non-quasianalytic.
			\item If $\om$ is a weight function,
			then
			\begin{equation}\label{newexpabsorb}
				\A h\ge 1 \E A\ge 1 \A x>0 \E D\ge 1 \A j\in\N : h^jW^{(x)}_j\le D W^{(Ax)}_j,
			\end{equation}
			which is crucial to have \eqref{omOMlocallyequal}.
		\end{enumerate}
	\end{lemma}

	\begin{proof}
		Cf.\ \cite[Section 5]{compositionpaper} and \cite[Section 2.5]{whitneyextensionweightmatrix}.
		For (3) see also \cite[Theorem 4.0.3, Lemma 5.1.3]{dissertation} and \cite[Lemma 2.5]{sectorialextensions}.
	\end{proof}

	\subsection{Order relations of weight matrices}
	For two weight matrices $\fM$ and $\fN$, we write
	$\fM (\preccurlyeq)\fN$ if
	for all $y$ there exists $x$ such that $M^{(x)} \preccurlyeq N^{(y)}$. By \cite[Proposition 4.6(1)]{compositionpaper},
	\[
	\cE^{(\fM)}(\R) \subseteq \cE^{(\fN)}(\R) \quad \Longleftrightarrow \quad \La^{(\fM)} \subseteq \La^{(\fN)} \quad \Longleftrightarrow \quad \fM (\preccurlyeq) \fN.
	\]
	If $\fM (\preccurlyeq) \fN$ and $\fN (\preccurlyeq) \fM$ hold simultaneously, then we say that $\fM$ and $\fN$ are \emph{equivalent}.
	This is the case if and only if $\cE^{(\fM)}(\R) = \cE^{(\fN)}(\R)$ as well as $\La^{(\fM)} = \La^{(\fN)}$ (as sets and, in turn, also as locally convex vector spaces).

	\begin{remark}
		Typically, for each notion of Beurling type there is a related version of Roumieu type.
		Since in this paper we are principally concerned with the Beurling case, we will only mention the former without emphasizing every time that it is the Beurling version.
	\end{remark}

	If $M$ is a weight sequence and $\fN$ a weight matrix, then $M \lhd N^{(x)}$ for all $x>0$ if and only if $\cE^{\{M\}}(\R) \subseteq \cE^{(\fN)}(\R)$; see \cite[Proposition 4.6(2)]{compositionpaper}.

	\subsection{Moderate growth and derivation closedness}\label{som1sect}

	For weight sequences $M,N$, consider
	\[
	mg(M,N) := \sup_{j+k \ge 1}\Big(\frac{M_{j+k}}{N_j N_k}\Big)^{\frac{1}{j+k}} \in (0,\infty]
	\]
	and
	\[
	dc(M,N) := \sup_{j\in \N}\left(\frac{M_{j+1}}{N_j}\right)^{\frac{1}{j+1}}  \in (0,\infty].
	\]
	A weight matrix $\fM = (M^{(x)})_{x>0}$ is said to have \emph{moderate growth}  if
	\begin{equation} \label{eq:mg}
		\tag{$\fM_{(mg)}$}
		\A y >0 \E x >0 : ~mg(M^{(x)}, M^{(y)})<\infty,
	\end{equation}
	and to be \emph{derivation closed} if
	\begin{equation} \label{eq:dc}
		\tag{$\fM_{(dc)}$}
		\A y >0 \E x >0:~dc(M^{(x)}, M^{(y)})<\infty.
	\end{equation}
	Note that moderate growth, derivation closedness, and non-quasianalyticity are
	preserved under equivalence.

	Derivation closedness allows for absorption of log-terms in associated weight functions:

	\begin{lemma}
		\label{lem:dcconsequence}
		Let $M^{(k)}$, for $1\le k \le l+1$, be weight sequences such that $dc(M^{(k)},M^{(k+1)})<\infty$ for all $1\le k\le l$. Then there exists $C >0$ such that
		\[
		\om_{M^{(l+1)}}(t)+\log(1+t^l)\le \om_{M^{(1)}}(Ct)+C, \quad t\ge 0.
		\]
	\end{lemma}

	\begin{proof}
		An iterated application of \cite[Lemma 2]{nuclearglobal2} yields the result.
	\end{proof}

	\subsection{Absorbing exponential growth}

	Inspired by \eqref{newexpabsorb} (cf.\ \cite[Section 4.1]{compositionpaper}),
	we say that a weight matrix $\fM$ \emph{absorbs exponential growth} if
	\begin{equation} \label{eq:exgrowth}
		\tag{$\fM_{(L)}$}
		\A y,h>0 \E x,A>0\A k \in \N:\quad h^kM^{(x)}_k\le A M^{(y)}_k.
	\end{equation}
	The weight matrix $\Om$ associated  with a weight function always has this property, by \Cref{Ompropcollect}.

	The following lemma states that for any weight matrix $\fM$ we may find an equivalent weight matrix with
	the property $(\fM_{(L)})$.
	For the sake of completeness, we mention that an analogous statement holds true in the Roumieu setting as well.

	\begin{lemma}
		\label{lem:ML}
		Let $\fM$ be a (normalized) weight matrix. Then there exists an equivalent (normalized) weight matrix $\fN$ that satisfies $(\fM_{(L)})$.
		Actually, we can choose $\fN$ such that for all $k\in\N_{\ge 1}$ there exist $A_k$ and $B_k$ such that for all $j\in\N$
		\begin{equation} \label{eq:equivalent}
			A_k\Big(\frac{1}{2^k}\Big)^j M^{(\frac{1}{k})}_j \le N_j^{(\frac{1}{k})} \le B_k\Big(\frac{1}{2^k}\Big)^j M^{(\frac{1}{k})}_j.
		\end{equation}
		Consequently, for all $t\ge 0$,
		\begin{equation}
			\label{eq:MLconsequence}
			\om_{M^{(1/k)}}(2^kt)-\log(B_k) \le \om_{N^{(1/k)} }(t) \le \om_{M^{(1/k)}}(2^kt)-\log(A_k).
		\end{equation}
	\end{lemma}

	\begin{proof}
		We will construct normalized weight sequences $N^{(\frac{1}{k})}$, indexed by $k \in \N_{\ge 1}$, satisfying \eqref{eq:equivalent} and
		$N^{(\frac{1}{k+1})} \le N^{(\frac{1}{k})}$  for all $k$.
		If we set $N^{(x)}:=N^{(\frac{1}{k})}$ for $\frac{1}{k+1}<x\le\frac{1}{k}$, then \eqref{eq:equivalent} implies that $\fM$ and $\fN$ are equivalent.
		Moreover, \eqref{eq:MLconsequence} follows from \eqref{eq:equivalent} and the definition of the associated weight function.
		To see that $\fN$ fulfills $(\fM_{(L)})$, fix $y$ and $h$ and choose $k,n \in \N$ such that $h \le 2^k$ and $\frac{1}{n}\le y$.
		By \eqref{eq:equivalent},
		\[
		h^j N^{(\frac{1}{k+n})}_j
		\le B_{k+n} \Big(\frac{1}{2^n}\Big)^j M^{(\frac{1}{k+n})}_j
		\le  B_{k+n} \Big(\frac{1}{2^n}\Big)^j M^{(\frac{1}{n})}_j
		\le \frac{B_{k+n}}{A_n} N^{(\frac{1}{n})}_j
		\le \frac{B_{k+n}}{A_n} N^{(y)}_j,
		\]
		for all $j$.

		Let us now construct the sequences $N^{(\frac{1}{k})}$.
		In the following, we work with $\mu^{(x)}_j = \frac{M^{(x)}_j}{M^{(x)}_{j-1}}$ and  $\nu^{(x)}_j = \frac{N^{(x)}_j}{N^{(x)}_{j-1}}$.
		Choose $j_0 \in \N$ minimal such that $\mu^{(1)}_j \ge 2$ for all $j \ge j_0$. Set
		\[
		\nu^{(1)}_j:= 1 \text{ for } j \le j_0, \quad \nu^{(1)}_j:= \frac{1}{2} \mu^{(1)}_j \text{ for } j > j_0,
		\]
		and $N^{(1)}_j:=\nu^{(1)}_0 \nu^{(1)}_1 \cdots \nu^{(1)}_j$, for $j\in \N$. Thus $N^{(1)}$ is clearly log-convex and satisfies \eqref{eq:equivalent} for $k = 1$.

		Now assume we have found sequences $N^{(\frac{1}{l})}$ such that \eqref{eq:equivalent} and $N^{(\frac{1}{l})} \le N^{(\frac{1}{l-1})}$ is satisfied for $l \le k$.
		Then we construct $N^{(\frac{1}{k+1})}$ as follows. Choose $j_0$ such that $\mu_j^{(\frac{1}{k+1})}\ge 2^{k+1}$ for all $j \ge j_0$.
		By \eqref{eq:equivalent} and the pointwise order of $\fM$, for $j \ge j_0$,
		\begin{align*}
			\Big(\frac{1}{2^{k+1}}\Big)^{j-j_0} \mu_{j_0+1}^{(\frac{1}{k+1})}\cdots \mu_{j}^{(\frac{1}{k+1})}
			&= \Big(\frac{1}{2^{k+1}}\Big)^{j-j_0} \frac{M^{(\frac{1}{k+1})}_{j}}{M^{(\frac{1}{k+1})}_{j_0}}
			\\
			&\le \frac{2^{(k+1)j_0}}{M^{(\frac{1}{k+1})}_{j_0}} \Big(\frac{1}{2^{k}}\Big)^{j} M^{(\frac{1}{k})}_{j}
			\\
			&\le \frac{2^{(k+1)j_0}}{A_k M^{(\frac{1}{k+1})}_{j_0}}  N^{(\frac{1}{k})}_{j}
			\\
			&= \frac{2^{(k+1)j_0} N_{j_0}^{(\frac{1}{k})}}{A_k M^{(\frac{1}{k+1})}_{j_0}}  \nu^{(\frac{1}{k})}_{j_0+1}\cdots \nu_j^{(\frac{1}{k})}
			=: B_k \nu^{(\frac{1}{k})}_{j_0+1}\cdots \nu_j^{(\frac{1}{k})}.
		\end{align*}
		Since $\mu_j^{(\frac{1}{k+1})} \to \infty$ as $j \to \infty$, there exists $j_1 >j_0$ such that
		\[
		\Big(\frac{1}{2^{k+1}}\Big)^{j_1-j_0} \mu_{j_0+1}^{(\frac{1}{k+1})}\cdots \mu_{j_1}^{(\frac{1}{k+1})} \ge B_k.
		\]
		Now set
		\[
		\nu^{(\frac{1}{k+1})}_j = 1 \text{ for }j \le j_1, \quad \nu^{(\frac{1}{k+1})}_j
		= \frac{1}{2^{k+1}}\mu_j^{(\frac{1}{k+1})} \text{ for } j > j_1.
		\]
		Then \eqref{eq:equivalent} is immediate.
		Combining the last two estimates, we also get $N^{(\frac{1}{k+1})} \le N^{(\frac{1}{k})}$.
		This ends the proof.
	\end{proof}

	\begin{corollary}
		For any weight matrix $\fM$ there is an equivalent weight matrix $\fN$ such that
		$\{\|\cdot\|^{N^{(1/k)}}_{[-k,k],1} :   k \in \N_{\ge 1}\}$ (resp.\ $\{\|\cdot\|^{N^{(1/k)}}_{1} :  k \in \N_{\ge 1}\}$)
		is a fundamental system of seminorms for $\cE^{(\fM)}(\R)$ (resp.\ $\La^{(\fM)}$).
	\end{corollary}

	\subsection{Strong $(\om_1)$ condition}

	Let us write $M \prec_{s\om_1} N$ if and only if
	\begin{equation}
		\label{som1}
		\tag{$s\om_1$}
		\E C>0 \A t \ge 0 :\quad \om_M(2t) \le \om_N(t)+C,
	\end{equation}
	and say that $M$ and $N$ satisfy the \emph{strong $(\om_1)$ condition}.

	Then another immediate consequence of \Cref{lem:ML} is the following.

	\begin{corollary}
		\label{rem:omega1wlog}
		Up to equivalence, we can assume that a weight matrix $\fM$ satisfies
		\begin{equation}\label{rem:omega1wlogequ}
			\A x>0 \E y>0 :~ M^{(x)} \prec_{s\om_1} M^{(y)}.
		\end{equation}
	\end{corollary}

	\begin{remark}  \label{rem:strongom1}
		In analogy to \eqref{om1}, one is led to the following condition:
		\begin{equation}\label{mixedbeuromega1}
			\A x>0 \E y>0  :\quad \om_{M^{(x)}}(2t)=O(\om_{M^{(y)}}(t)) \text{ as } t \to \infty;
		\end{equation}
		see \cite{testfunctioncharacterization} and \cite{mixedgrowthindex}.
		But \eqref{rem:omega1wlogequ} is stronger than \eqref{mixedbeuromega1}. Cf.\ the results from \cite[Section 3]{mixedgrowthindex} and the citations therein as well as \Cref{constantsinout}.
	\end{remark}

	\section{Reduction to the Roumieu case}
	\label{sec:reduction}

	The goal of this section is to prove the following theorem.

	\begin{theorem}
		\label{thm:mainthm1}
		Let $\fM, \fN$ be weight matrices that are ordered with respect to their quotient sequences, i.e., $\mu^{(x)}\le \mu^{(y)}$ and $\nu^{(x)}\le \nu^{(y)}$ if $x \le y$.
		Then
		\begin{equation}
			\label{eq:SVchar}\tag{SV}
			\La^{(\fM)} \subseteq j^\infty_0 \cE^{(\fN)}(\R) \quad \Longleftrightarrow \quad \forall y>0 \E x >0:~  M^{(x)} \prec_{SV} N^{(y)}.
		\end{equation}
	\end{theorem}

	We shall see in \Cref{assumptionssuperfluous} that
	both sides of the equivalence \eqref{eq:SVchar} imply
	$\fM (\preccurlyeq) \fN$ and non-quasianalyticity of $\fN$.
	Recall that $M \prec_{SV} N$ means
	\begin{equation}
		\label{eq:SV}
		\E C,s \in \N_{\ge 1}:~
		\sup_{j\ge 1}\sup_{0\le i < j} \Big(\frac{M_j}{s^j N_i} \Big)^{\frac{1}{j-i}}\frac{1}{j}\sum_{k = j}^\infty \frac{N_{k-1}}{N_k} \le C.
	\end{equation}

	We will deduce \Cref{thm:mainthm1} from the following result for Denjoy--Carleman classes of Roumieu type.
	It is due to \cite{surjectivity} under slightly stronger conditions;
	the version stated here is a special case of \cite[Theorem 3.2]{mixedramisurj}.

	\begin{theorem}
		\label{thm:SVMain}
		Let $M \preccurlyeq N$ be weight sequences with $\liminf_{p \rightarrow \infty} \left( m_p \right)^{1/p} > 0$. Then
		$\La^{\{M\}} \subseteq j^\infty_0 \cD^{\{N\}}([-1,1])$
		if and only if
		$M \prec_{SV} N$.
	\end{theorem}

	Here $\cD^{(N)}([-1,1])$ (resp.\ $\cD^{\{N\}}([-1,1])$) denotes the space of $\cE^{(N)}$ (resp.\ $\cE^{\{N\}}$) functions supported in $[-1,1]$.

	\subsection{Auxiliary results}
	We show first that both sides of the equivalence \eqref{eq:SVchar} imply $\fM (\preccurlyeq) \fN$ and non-quasianalyticity of $\fN$.
	Similar results hold
	in the  Roumieu case; cf. \cite{RoumieuMoment}.

	\begin{lemma}\label{assumptionssuperfluous}
		Let $\fM$ and $\fN$ be weight matrices. Both sides of the equivalence \eqref{eq:SVchar} imply
		$\fM (\preccurlyeq) \fN$ and non-quasianalyticity of $\fN$.
	\end{lemma}

	\begin{proof}
		By \cite[Lemma 3.2]{maximal},
		$M^{(x)} \prec_{SV} N^{(y)}$ implies $M^{(x)} \preccurlyeq N^{(y)}$
		so that $\fM (\preccurlyeq) \fN$ is clearly a consequence of the right-hand side of \eqref{eq:SVchar}.

		To see that it also follows from the left-hand side,
		suppose that $\fM (\preceq) \fN$ is violated which means
		that there is $y>0$ such that $(M^{(x)}_k/N^{(y)}_k)^{1/k}$ is unbounded for all $x>0$.
		Thus, for all $j \in \N_{\ge 1}$ we find $k_j \ge j$ such that
		\[
		\Big(\frac{M^{(1/j)}_{k_j}}{N^{(y)}_{k_j}}\Big)^{1/k_j} \ge j.
		\]
		Consider the sequence $a=(a_{\ell})$ with $a_{k_j} = (\frac{1}{j})^{k_j} M_{k_j}^{(1/j)}$ and $a_\ell = 0$ otherwise.
		Then $a \in\Lambda^{(\fM)}$, because for given $h,z>0$ and
		$j$ so large that $\frac{1}{j}\le \min\{h, z\}$,
		we have
		$|a_{k_j}| =(\tfrac{1}{j})^{k_j} M_{k_j}^{(1/j)}\le h^{k_j}M^{(z)}_{k_j}$.
		On the other hand, we claim that $a\notin j^{\infty}_0 \cE^{(\fN)}(\R)$.
		Indeed, if there is $f\in\mathcal{E}^{(\fN)}(\R)$
		with $j^{\infty}_0 f=a$, then
		$N^{(y)}_{k_j}\le a_{k_j}=f^{(k_j)}(0)\le A_{h,z} h^{k_j}N^{(z)}_{k_j}$
		for all $h,z>0$ and $j$;
		a contradiction for $z = y$ and $h=1/2$.

		To infer non-quasianalyticity of $\fN$, we distinguish two cases.
		If $\fM$ is non-quasianalytic
		so is $\fN$, since we already know that $\fM (\preccurlyeq) \fN$.
		If $\fM$ is quasianalytic, then the assertion follows either from \cite[Theorem 6]{borelmappingquasianalytic},
		which shows that no (proper) quasianalytic class
		is contained in the image of the Borel map of any other quasianalytic class,
		or from the observation that $M^{(x)} \prec_{SV} N^{(y)}$
		cannot hold if $N^{(y)}$ is quasianalytic (since then \eqref{eq:SV} is infinite).
	\end{proof}

	We restate \cite[Lemme 16]{ChaumatChollet94} which is crucial for the reduction.

	\begin{lemma}\label{Beurling-Surjectivitytheorem1}
		Let $(\alpha_j)$ be a sequence of nonnegative real numbers such that $\sum_{j=1}^{\infty}\alpha_j<\infty$.
		Let $(\beta_j)$ and $(\gamma_j)$ be sequences of positive real numbers such that
		$\lim_{j\rightarrow\infty}\beta_j=0=\lim_{j\rightarrow\infty}\gamma_j$, and assume that $(\gamma_j)$ is decreasing.
		Then there exists an increasing sequence $(\theta_j)$ tending to $\infty$ such that
		\begin{enumerate}[label=\thetag{\arabic*}]
			\item $\theta_j\gamma_j$ is decreasing,
			\item $\theta_j\beta_j \to 0$,
			\item $\sum_{k=j}^{\infty}\theta_k\alpha_k\le 8\theta_j\sum_{k=j}^{\infty}\alpha_k$ for all $j \ge 1$.
		\end{enumerate}
	\end{lemma}

	\subsection{Scheme of proof}
	\label{sec:scheme}
	The direction $\Rightarrow$ in  \eqref{eq:SVchar}
	follows from a rather direct generalization of the proof of \cite[Theorem 4.7]{mixedramisurj} which we
	sketch in \Cref{sec:right}.

	The more delicate part is the converse implication. Our aim is to reduce its proof to the Roumieu case, i.e., \Cref{thm:SVMain}. More specifically,
	we show that for any given $\la \in \Lambda^{(\fM)}$ we find weight sequences $R,S$ such that
	\begin{enumerate}[label=\thetag{\roman*}]
		\item $\la \in \Lambda^{\{R\}}$,
		\item $R \prec_{SV} S$,
		\item $\cE^{\{S\}}(\R) \subseteq \cE^{(\fN)}(\R)$.
	\end{enumerate}
	Then \Cref{thm:SVMain} (together with \Cref{assumptionssuperfluous}) gives the desired conclusion.

	\subsection{Proof of \Cref{thm:mainthm1}($\Leftarrow$)} \label{sec:left}

	We construct the sequences $R,S$ in several steps.

	\begin{step*}[I]
		Up to equivalence, we may assume that $\fM$ and $\fN$ satisfy the following conditions:
		\begin{enumerate}[label=\thetag{\alph*}]
			\item For all $\al \in \N_{\ge 1}$,
			\begin{equation}\label{liminfimportant}
				N^{(\frac{1}{\al})}_j \ge 2^j N_j^{(\frac{1}{\al+1})}, \quad \text{ for large enough } j.
			\end{equation}
			\item For all $y>0$ we have $M^{(y)} \prec_{SV} N^{(y)}$ with $C=s =1$ in \eqref{eq:SV}.
			\item For all $y>0$ we have $M^{(y)}\le N^{(y)}$.
		\end{enumerate}
	\end{step*}

	\begin{proof}
		(a) follows from \Cref{lem:ML}.

		(b), (c) Fix $y>0$. By assumption, there is $x=x(y)$ such that $M^{(x)} \prec_{SV} N^{(y)}$ and thus
		$M^{(x)} \preceq N^{(y)}$. We may assume that $y \mapsto x(y)$ is increasing and $x(y) \le y$ (by the order of the weight matrices).
		Then $(M^{(x(y))})_{y>0}$ is equivalent to $\fM$.
		Finally, there exists an increasing function $r$ with $r(0)=0$ such that the family $\fM'$ with $M'^{(y)}_j:= r(y)^j M^{(x(y))}_j$ satisfies the additional assumption of (b) and (c).
		This matrix is not normalized, but we can use an analogous technique as in the proof of \Cref{lem:ML} to force this as well.

		Note that all constructions yield matrices that are still ordered with respect to their quotients.
	\end{proof}

	We assume from now on that $\fM$ and $\fN$ satisfy (a),(b), and (c).
	Fix $\la \in \La^{(\fM)}$.

	\begin{step*}[II]
		There exist a decreasing 0-sequence $(\ve_j)$ and a strictly increasing sequence of positive integers $(a_\al)$ such that
		\begin{equation}
			\label{eq:epsdef}
			|\lambda_j|\le\varepsilon_1\cdots\varepsilon_jM^{(\frac{1}{\alpha+1})}_j, \quad \text{ if }~ a_{\alpha}\le j< a_{\alpha+1}.
		\end{equation}
	\end{step*}

	\begin{proof}
		By definition of $\La^{(\fM)}$, the sequence $\ve^{(\al)}:= (\ve^{(\al)}_j)$ defined by
		\[
		\varepsilon^{(\alpha)}_j := \sup_{k\ge j}\Big(\frac{|\lambda_k|}{M^{(\frac{1}{\alpha+1})}_k}\Big)^{1/k}
		\]
		is decreasing and tending to $0$ for each $\al$. By the order of $\fM$, we also have $\ve^{(\al)}\le \ve^{(\al+1)}$.
		We define sequences $(a_{\al})$ and $(a'_{\al})$ of positive integers as follows:
		\begin{itemize}
			\item Set $a_1:=1$.
			\item For given $a_{\alpha}$, we choose $a'_{\alpha}$ and in turn $a_{\alpha+1}$ such that
			\[
			\varepsilon^{(\alpha+1)}_{a_{\alpha+1}}< \varepsilon^{(\alpha)}_{a'_{\alpha}}\le\frac{1}{1+\alpha}\varepsilon^{(\alpha)}_{a_{\alpha}}.
			\]
		\end{itemize}
		It is clear that the sequences $(a_{\al})$ and $(a'_{\al})$ are strictly increasing and interlacing.
		Finally, define $\ve = (\ep_j)$ by
		\begin{equation*}
			\varepsilon_j:=\varepsilon^{(\alpha)}_j\;\;\;\text{for}\;a_{\alpha}\le j\le a'_{\alpha},\hspace{20pt}\varepsilon_j:=\varepsilon^{(\alpha)}_{a'_{\alpha}}\;\;\;\text{for}\;a'_{\alpha}<j<a_{\alpha+1}.
		\end{equation*}
		Then $\ve$ is decreasing, tending to $0$, and, by construction
		\begin{equation*}
			\Big(\frac{|\lambda_j|}{M^{(\frac{1}{\alpha+1})}_j}\Big)^{1/j}\le\varepsilon_j, \quad \text{ if }~ a_{\alpha}\le j< a_{\alpha+1},
		\end{equation*}
		which gives \eqref{eq:epsdef}.
	\end{proof}

	\begin{step*}[III]
		There exist an increasing sequence $(\ul \mu_j)$ with $\ul \mu_j/j \to \infty$, and strictly increasing sequences of integers $(b_\al)$ and $(C_\al)$
		such that $\ul M_j := \ul \mu_0 \ul \mu_1 \cdots \ul \mu_j$ satisfies $\ul M \le C_\al M^{(1/\alpha)}$, for all $\al$, and
		\begin{gather} \label{eq:ulMMal}
			M^{(1/\alpha)}_j\le\underline{M}_j, \quad \text{ for all } \al \text{ and } j\le b_{\alpha}.
		\end{gather}
	\end{step*}

	\begin{proof}
		Let $(a_\al)$ be the sequence from Step (II).
		We define sequences of positive integers $(b_{\alpha})$ and $(b'_{\alpha})$ as follows:

		\begin{itemize}
			\item Set $b'_1:=1$.
			\item For given $b'_{\alpha}$, we choose $b_{\alpha}$ such that
			\begin{equation}\label{underlineMbelowunderlineNassump}
				b_{\alpha}>\max\{a_{\alpha},b'_{\alpha}\} \quad \text{ and } \quad \mu^{(\frac{1}{\alpha+1})}_{j}\ge\alpha j, \quad \text{ for } j\ge b_{\alpha}.
			\end{equation}
			\item For given $b_{\alpha}$, we choose $b'_{\alpha+1}>b_{\alpha}$ minimal to ensure
			\[
			\mu^{(\frac{1}{\alpha+1})}_{b'_{\alpha+1}}>\mu^{(\frac{1}{\al})}_{b_{\alpha}}.
			\]
		\end{itemize}
		Note that $(b_{\alpha})$ and $(b'_{\alpha})$ are strictly increasing, interlacing, and $\mu^{(\frac{1}{\alpha+1})}_j\le\mu^{(\frac{1}{\al})}_{b_{\alpha}}$ for all $j\le b'_{\alpha+1}-1$.
		Finally, set
		\begin{equation}\label{muunderlinedef}
			\underline{\mu}_j:=\mu^{(\frac{1}{\al})}_j,\quad \text{ for } b'_{\alpha}\le j\le b_{\alpha},
			\qquad \underline{\mu}_j:=\mu^{(\frac{1}{\al})}_{b_{\alpha}},\quad \text{ for } b_{\alpha}<j<b'_{\alpha+1},
		\end{equation}
		and $\ul \mu_0:=1$.
		By construction, $\ul \mu_j$ is increasing, $\ul \mu_j/j \to \infty$, and \eqref{eq:ulMMal} holds.
		For fixed $\al$, one has $\ul \mu_j\le \mu_j^{(1/\al)}$ for all $j \ge b_\al'$
		which yields $\ul M \le C_\al M^{(1/\al)}$ for some positive constant $C_\al$.
		Clearly, we may assume that $C_\al$ are integers, strictly increasing in $\al$.
	\end{proof}

	\begin{step*}[IV]
		There exist an increasing sequence $(\ul \nu_j)$ tending to $\infty$,
		strictly increasing sequences of positive integers
		$(c_\al)$ and $(d_\al)$, and an increasing sequence $(D_\al)$ tending to $\infty$ such that
		$\ul N_j := \ul \nu_0 \ul \nu_1 \cdots \ul \nu_j$ satisfies
		$\ul N \le D_\al N^{(1/\alpha)}$, for all $\al$,
		\begin{gather} \label{eq:IV1}
			N^{(1/\alpha)}_j\le\underline{N}_j,  \quad \text{ for all } \al \text{ and } j\le d_{\alpha},
		\end{gather}
		and there is a constant $D \ge 1$ such that, for all $\al$ and $c_{\alpha}\le j < c_{\alpha+1}$,
		\begin{gather} \label{eq:IV2}
			\sum_{k \ge j} \frac{1}{{\ul \nu}_k} \le 2 \sum_{k \ge j} \frac{1}{\nu_k^{(\frac{1}{\al+2})}},
			\\ \label{eq:IV3}
			C_{\alpha+3} N^{(\frac{1}{\alpha+3})}_i\le D2^{j-i}\underline{N}_i, \quad \text{ for all }~ 0\le i< j,
		\end{gather}
		where $C_\al$ are the constants from Step (III).
	\end{step*}

	\begin{proof}
		We define sequences $(c_{\al})$ and $(d_{\al})$ of positive integers as follows:
		\begin{itemize}
			\item Set $c_1:=1$ and $d_0:= 0$.
			\item For given $c_{\alpha}$, we choose $d_{\alpha} \ge C_{\al+4} + d_{\al-1}$ such that
			\begin{gather}\label{nuunderlinedef1}
				\sum_{k> d_{\alpha}}\frac{1}{\nu^{(\frac{1}{\alpha+1})}_k}\le\frac{1}{2}\sum_{k> c_{\alpha}}\frac{1}{\nu^{(\frac{1}{\al})}_k},
				\\ \label{nuunderlinedef10}
				N^{(\frac{1}{\alpha+2})}_{j} \ge 2^j N^{(\frac{1}{\alpha+3})}_j, \quad \text{ for } j\ge d_{\alpha}.
			\end{gather}
			\item For given $d_{\alpha}$, we choose $c_{\alpha+1}>d_{\alpha}$ minimal such that
			\begin{equation}\label{nuunderlinedef2}
				\nu^{(\frac{1}{\alpha+1})}_{c_{\alpha+1}}>\nu^{(\frac{1}{\alpha})}_{d_{\alpha}}.
			\end{equation}
		\end{itemize}
		Then $(c_\al)$ and $(d_\al)$ are strictly increasing and interlacing.
		Set
		\begin{equation}\label{nuuunderlinedef}
			\underline{\nu}_j:=\nu^{(\frac{1}{\al})}_j,\quad c_{\alpha}\le j\le d_{\alpha},
			\qquad \underline{\nu}_j:=\nu^{(\frac{1}{\al})}_{d_{\alpha}},\quad d_{\alpha} < j<  c_{\alpha+1},
		\end{equation}
		and $\underline{\nu}_0:=1$.
		Completely analogous to Step (III), we may conclude that $\ul N \le D_\al N^{(1/\alpha)}$ and \eqref{eq:IV1}.

		Let us show \eqref{eq:IV3}. It clearly suffices to show the claim for $\al \ge 3$
		(the finitely many remaining values can be controlled by possibly enlarging $D$).
		So let $\al \ge 3$, $c_\al \le j < c_{\al+1}$, and $0\le i < j$.
		Our construction yields $j \ge c_\al > d_{\al-1} \ge C_{\al+3}$, and therefore
		\[
		C_{\al+3} N^{(\frac{1}{\alpha+3})}_0 = C_{\al+3}\le 2^{d_{\al-1}} \le 2^{j}=2^j{\ul N}_0
		\]
		which finishes the case $i = 0$. So let $1\le i < j$.
		There is $\be\le \al$ such that $c_\be \le i < c_{\be+1}$. If $\beta\ge 2$, then $d_{\beta-1}\le i\le d_{\beta+1}$.
		By \eqref{eq:IV1} and \eqref{nuunderlinedef10},
		\[
		\frac{\underline{N}_i}{N^{(\frac{1}{\alpha+3})}_i}\ge\frac{N^{(\frac{1}{\beta+1})}_i}{N^{(\frac{1}{\alpha+3})}_i}
		\ge\frac{N^{(\frac{1}{\beta+1})}_i}{N^{(\frac{1}{\beta+2})}_i}\ge 2^i.
		\]
		Since $2^j \ge d_{\alpha-1}\ge C_{\alpha+3}$, \eqref{eq:IV3} follows.
		It remains to consider $1\le i\le d_1$, in which case
		$N^{(\frac{1}{\alpha+3})}_i\le N^{(1)}_i=\underline{N}_i$ is clear and $2^{j-i}\ge 2^{d_{\alpha-1}-d_1}\ge C_{\alpha+3}$.
		Thus \eqref{eq:IV3} is proved.

		Let us now prove \eqref{eq:IV2}. First assume $c_{\alpha}\le j\le d_{\alpha}$.  Then
		\[
		\sum_{k\ge j}\frac{1}{\underline{\nu}_k}=\sum_{k=j}^{d_{\alpha}}\frac{1}{\nu^{(\frac{1}{\alpha})}_k}
		+\sum_{i\ge 1} \Big( \frac{c_{\alpha+i}-d_{\alpha+i-1}-1}{\nu^{(\frac{1}{\alpha+i-1})}_{d_{\alpha+i-1}}}
		+\sum_{k=c_{\alpha+i}}^{d_{\alpha+i}}\frac{1}{\nu^{(\frac{1}{\alpha+i})}_k} \Big).
		\]
		By the minimal choice of $c_{\al+i}$ (see \eqref{nuunderlinedef2}),
		\begin{equation}\label{nuunderlinedef4}
			\frac{c_{\alpha+i}-d_{\alpha+i-1}-1}{\nu^{(\frac{1}{\alpha+i-1})}_{d_{\alpha+i-1}}}
			\le
			\sum_{k=d_{\alpha+i-1}+1}^{c_{\alpha+i}-1}\frac{1}{\nu_{k}^{(\frac{1}{\alpha+i})}},
		\end{equation}
		whence
		\begin{align*}
			\sum_{k\ge j}\frac{1}{\underline{\nu}_k}
			&\le\sum_{k=j}^{d_{\alpha}}\frac{1}{\nu^{(\frac{1}{\alpha+1})}_k}
			+\sum_{i\ge 1} \sum_{k=d_{\alpha+i-1}+1}^{d_{\alpha+i}} \frac{1}{\nu_{k}^{(\frac{1}{\alpha+i})}}
			\\&
			=\sum_{k=j}^{d_{\alpha+1}}\frac{1}{\nu^{(\frac{1}{\alpha+1})}_k}+\sum_{i\ge 2}\sum_{k=d_{\alpha+i-1}+1}^{d_{\alpha+i}}\frac{1}{\nu_{k}^{(\frac{1}{\alpha+i})}}.
		\end{align*}
		Using \eqref{nuunderlinedef1},
		we find
		\begin{align*}
			\sum_{k\ge d_{\alpha+i-1}+1}\frac{1}{\nu_{k}^{(\frac{1}{\alpha+i})}}
			\le \frac{1}{2^{i-1}}\sum_{k\ge d_{\alpha}+1}\frac{1}{\nu_{k}^{(\frac{1}{\alpha+1})}}
		\end{align*}
		from which it is easy to conclude
		\begin{align}
			\label{eq:estimate1}
			\sum_{k\ge j}\frac{1}{\underline{\nu}_k}
			&\le 2\sum_{k\ge j}\frac{1}{\nu^{(\frac{1}{\alpha+1})}_k},
		\end{align}
		in particular, \eqref{eq:IV2}.
		If $d_{\alpha}<j<c_{\alpha+1}$, then, using \eqref{eq:estimate1} for $j = c_{\al+1}$, we find
		\[
		\sum_{k\ge j}\frac{1}{\underline{\nu}_k}
		= \sum_{k=j}^{c_{\alpha+1}-1}\frac{1}{\nu^{(\frac{1}{\al})}_{d_{\alpha}}}+\sum_{k\ge c_{\alpha+1}}\frac{1}{\underline{\nu}_k}
		\le \sum_{k=j}^{c_{\alpha+1}-1}\frac{1}{\nu^{(\frac{1}{\alpha+2})}_k}+2\sum_{k\ge c_{\al+1}}\frac{1}{\nu^{(\frac{1}{\alpha+2})}_k}
		\le 2\sum_{k\ge j}\frac{1}{\nu^{(\frac{1}{\alpha+2})}_k}.
		\]
		Thus \eqref{eq:IV2} is proved.
	\end{proof}

	\begin{step*}[V]
		There exist weight sequences $R, S$ such that $(r_j)^{1/j} \to \infty$ and
		\begin{enumerate}[label=\thetag{\roman*}]
			\item $\la \in \Lambda^{\{R\}}$,
			\item $R \prec_{SV} S$,
			\item $\cE^{\{S\}}(\R) \subseteq \cE^{(\fN)}(\R)$.
		\end{enumerate}
	\end{step*}

	\begin{proof}
		For the construction of $R$, we
		apply \Cref{Beurling-Surjectivitytheorem1} to
		\[
		\alpha_j:=0, \quad \beta_j:=\max\Big\{\varepsilon_j,\frac{j}{(\underline{M}_j)^{1/j}}\Big\}, \quad \gamma_j:=\frac{1}{\underline{\mu}_j}.
		\]
		This yields an increasing sequence $(\theta_j)$ tending to $\infty$ such that $\theta_j\gamma_j$ is decreasing and $\theta_j\beta_j\to 0$.
		We can assume $\theta_0=1$. Since  $\th_j \ga_j \le \th_j \be_j$ (as $(\underline{M}_j)^{1/j} \le \ul \mu_j$), also $\theta_j\ga_j\to 0$. Then
		\[
		R_j:=\prod_{i = 0}^j \frac{{\ul \mu}_j}{\th_j} = \frac{\ul M_j}{\theta_0 \th_1\cdots\theta_j}
		\]
		is a weight sequence (not necessarily normalized). We have
		$(r_j)^{1/j} \to \infty$,
		since
		\[
		\frac{j}{(R_j)^{1/j}}
		=\frac{j (\theta_1\cdots\theta_j)^{1/j}}{(\underline{M}_j)^{1/j}}\le\frac{j\theta_j}{(\underline{M}_j)^{1/j}}\le\theta_j\beta_j.
		\]
		By \eqref{eq:epsdef}, \eqref{eq:ulMMal}, and \eqref{underlineMbelowunderlineNassump},
		\[
		|\la_j| \le \ve_1\cdots\ve_j M_j^{(\frac{1}{\al+1})} \le \ve_1\th_1 \cdots \ve_j\th_j R_j.
		\]
		Since $\ve_j\th_j\le \be_j\th_j \rightarrow 0$, we get $\la \in \Lambda^{\{R\}}$. This finishes the proof of (i).

		To obtain $S$ we
		apply \Cref{Beurling-Surjectivitytheorem1} to
		\[
		\alpha'_j=\gamma'_j:=\frac{1}{\underline{\nu}_j},\quad  \beta'_j:=\max\Big\{\frac{1}{\sqrt{\theta_{\lfloor j/2\rfloor}}},\frac{1}{\underline{\nu}_j}\Big\},
		\]
		where $\lfloor j/2\rfloor$ denotes the integer part of $j/2$.
		We obtain an increasing sequence $(\theta'_j)$ tending to $\infty$
		such that $\theta'_j\gamma'_j$ is decreasing, $\theta'_j\beta'_j\to 0$, and
		\begin{equation}\label{Beurling-equ2}
			\sum_{k=j}^{\infty}\frac{\theta'_k}{\underline{\nu}_k}\le 8\theta'_j\sum_{k=j}^{\infty}\frac{1}{\underline{\nu}_k}, \quad \text{ for all }j.
		\end{equation}
		Let $\th'_0:=1$.
		Then
		\[
		S_j:=A^j \prod_{i = 0}^j \frac{{\ul \nu}_j}{\th_j'} = A^j \frac{\ul N_j}{\th_0' \th_1' \cdots \th_j'}
		\]
		is a weight sequence.
		Here $A$ is a constant chosen such that $A\ge \max\{1,\frac{\th'_1}{{\ul \nu}_1}\}$ and
		\begin{gather}
			\label{eq:th'th1}
			\frac{\th'_i}{\th_i} \le A,
			\\
			\label{eq:th'th2}
			\frac{\th'_j}{(\th_{i+1} \cdots \th_j)^{\frac{1}{j-i}}} \le A, \quad \text{ if } 0 \le i < j.
		\end{gather}
		That \eqref{eq:th'th1} and \eqref{eq:th'th2} are possible is seen as follows.
		It is easy to see that the choice of $\be'_j$ enables \eqref{eq:th'th1}.
		For $0\le i< \lfloor j/2\rfloor$, we have
		\begin{align*}
			(\theta_{i+1}\cdots\theta_j)^{\frac{1}{j-i}}
			\ge(\theta_{\lfloor j/2\rfloor}\cdots\theta_j)^{\frac{1}{j-i}}
			\ge(\theta_{\lfloor j/2\rfloor})^{\frac{j-\lfloor j/2\rfloor}{j-i}}
			\ge\sqrt{\theta_{\lfloor j/2\rfloor}},
		\end{align*}
		since $\th_j$ is increasing. If $\lfloor j/2\rfloor\le i \le j-1$, then $(\theta_{i+1}\cdots\theta_j)^{1/(j-i)}\ge\theta_{i+1}\ge\theta_{\lfloor j/2\rfloor}$.
		The choice of $\be'_j$ shows that the left-hand side of \eqref{eq:th'th2} is bounded.

		Let us now show that $R \prec_{SV} S$.
		Fix $0\le i < j$.
		There is $\al$ such that $c_\al \le j < c_{\al+1}$.  We have (with $\si_k = S_k/S_{k-1}$)
		\begin{align*}
			\sum_{k=j}^{\infty}\frac{1}{\sigma_k}
			& \stackrel{\eqref{Beurling-equ2}}{\le}  \frac{8}{A} \theta'_j\sum_{k=j}^{\infty}\frac{1}{\underline{\nu}_k}
			\stackrel{\eqref{eq:IV2}}{\le} \frac{16}{A} \theta'_j \sum_{k=j}^{\infty}\frac{1}{\nu^{(\frac{1}{\alpha+2})}_k}
			\le \frac{16}{A} \theta'_j \sum_{k=j}^{\infty}\frac{1}{\nu^{(\frac{1}{\alpha+3})}_k}
			\\&
			\stackrel{\thetag{I}_b}{\le} \frac{16}{A} j \theta'_j \Big(\frac{N^{(\frac{1}{\alpha+3})}_i}{M^{(\frac{1}{\alpha+3})}_j}\Big)^{\frac{1}{j-i}}
			\stackrel{\thetag{III}}{\le} \frac{16}{A} j \theta'_j\Big(\frac{C_{\alpha+3}N^{(\frac{1}{\alpha+3})}_i}{\underline{M}_j}\Big)^{\frac{1}{j-i}}
			\\
			&\stackrel{\eqref{eq:IV3}}{\le} \frac{32D}{A} j \theta'_j \Big(\frac{\underline{N}_i}{\underline{M}_j}\Big)^{\frac{1}{j-i}}
			= \frac{32D}{A} j \theta'_j \Big(\frac{\theta'_1\cdots\theta'_iS_i}{A^i\theta_1\cdots\theta_ j R_j}\Big)^{\frac{1}{j-i}}\\
			&\stackrel{\eqref{eq:th'th1}}{\le}
			\frac{32D}{A} j \theta'_j \Big(\frac{S_i}{\theta_{i+1}\cdots\theta_ j R_j}\Big)^{\frac{1}{j-i}}
			\stackrel{\eqref{eq:th'th2}}{\le}
			32D j \Big(\frac{S_i}{R_j}\Big)^{\frac{1}{j-i}},
		\end{align*}
		which finishes the proof of (ii).

		For (iii) observe that, by Step (IV),
		\begin{align*}
			\frac{S_j}{N^{(1/\al)}_j} = \frac{A^j}{\th_1' \cdots \th_j'} \frac{\ul N_j}{ N^{(1/\al)}_j}
			\le D_\al  \frac{A^j}{\th_1' \cdots \th_j'}.
		\end{align*}
		We conclude that $S \lhd N^{(1/\al)}$ for all $\al$, since $\th_j'\to \infty$.
	\end{proof}

	\begin{step*}[VI]
		There exists $f \in \cE^{(\fN)}(\R)$ such that $j^\infty_0 f=\la$.
	\end{step*}

	\begin{proof}
		By Step (V) and \cite[Lemma 3.2]{maximal}, \Cref{thm:SVMain} can be applied to $R$ and $S$.
		Thus there exists $f \in \cE^{\{S\}}(\R)$ with $j^\infty_0f=\la$. By $\thetag{V}_{iii}$, we know that $f \in \cE^{(\fN)}(\R)$.
	\end{proof}

	\subsection{Proof of \Cref{thm:mainthm1}($\Rightarrow$)} \label{sec:right}

	Since $\La^{(\fM)} \subseteq j^\infty_0 \cE^{(\fN)}(\R)$ implies that $\fN$ is non-quasianalytic,
	by \Cref{assumptionssuperfluous}, and so there exist $\cE^{(\fN)}(\R)$-cutoff functions,
	e.g.\ by \cite[Corollary 3.2 and Theorem 11.16]{Rainer:2021aa},
	we have  $\La^{(\fM)} \subseteq j^\infty_0 \cD^{(\fN)}([-1,1])$.

	Now we follow the ideas of \cite[Proposition 4.3 and Theorem 4.4]{surjectivity}.
	Let $E_{m,k}$ be $\cD^{(\fN)}([-1,1])$ endowed with the norm $ \|f\|_{m,k} := \|f\|_{[-1,1],\frac{1}{m}}^{N^{(\frac{1}{k})}}$
	and let $F_{m,k}$ be its completion.
	As in \cite[Proposition 4.3]{surjectivity}, one sees that, for all $m,k$, there exists a continuous linear right inverse
	$T_{m,k}: \Lambda^{(\fM)} \rightarrow F_{m,k}$  of $j^\infty_0|_{F_{m,k}}$.
	Then for every $m \in \N_{\ge 1}$ we can find $s \in \N_{\ge 1}$ and $C>0$ such that
	\[
	\|T_{m,1}(a)\|_{m,1} \le C \|a\|_s, \quad a\in \Lambda^{(\fM)},
	\]
	where $\|\cdot\|_s:=\|\cdot\|^{M^{(\frac{1}{s})}}_{\frac{1}{s}}$.
	The proof of \cite[Theorem 4.4]{surjectivity}, applied to the sequences $M^{(\frac{1}{s})}$ and $N^{(\frac{1}{m})}$,
	yields
	$M^{(\frac{1}{s})} \prec_{SV} N^{(\frac{1}{m})}$.
	This ends the proof of \Cref{thm:mainthm1}.

	\section{The duals of $\cE^{(\fM)}(\R)$ and $\La^{(\fM)}$} \label{sec:duals}

	In this section, we identify the duals of $\cE^{(\fM)}(\R)$ and $\La^{(\fM)}$ with weighted spaces of entire functions.
	This will be of crucial importance in the proof of the second main result, i.e., \Cref{thm:mainthm2}.

	\subsection{Weighted spaces of entire functions}

	Let $g:\C \rightarrow [0,\infty)$ be a continuous function with $\lim_{|z|\rightarrow\infty}g(z)=\infty$.
	We define the Banach space
	\[
	A_g:= \Big\{f\in \cH(\C): ~ \|f\|_{A_g}:= \sup_{z \in \C}\frac{|f(z)|}{e^{g(z)}} <\infty \Big\}.
	\]
	Given an increasing sequence of continuous functions $\cG= (g_k)_k$ of the mentioned type, we define
	\[
	\cA_{\cG}:= \bigcup_{k\in \N} A_{g_k},
	\]
	and endow it with the natural inductive limit topology.
	Sometimes we need to work with $L^2$ weights.
	To this end, we set
	\[
	A_g^2:= \Big\{f\in \cH(\C): ~ \|f\|_{A_g^2}:= \Big(\int_{\C}|f(z)|^2e^{-g(z)} \,d\la (z) \Big)^{1/2} <\infty \Big\},
	\]
	where $\la$ denotes the Lebesgue measure in $\C$,
	and define the corresponding inductive limit $\cA^2_{\cG}$ analogously.

	\begin{lemma}
		\label{lem:Apincl}
		Let $g:\C \rightarrow [0,\infty)$ be a continuous function with $\lim_{|z|\rightarrow\infty}g(z)=\infty$. Then
		\[
		\|f\|_{A^2_{2g+\log(1+|z|^4)}} \le 3\pi \|f\|_{A_g}, \quad f \in A_g,
		\]
		in particular, $A_g \hookrightarrow A^2_{2g+\log(1+|z|^4)}$.

		Let $h:\C \rightarrow [0,\infty)$ be another continuous function and assume there exists $K>0$ such that
		\begin{equation}
			\label{eq:cons1}
			g(z+u) \le h(z) + K, \quad  z,u \in \C,~|u|\le 1.
		\end{equation}
		Then
		\[
		\|f\|_{A_{\frac{h}{2}}} \le e^{K} \|f\|_{A^2_g}, \quad f \in A^2_g,
		\]
		in particular,
		$A^2_g \hookrightarrow A_{\frac{h}{2}}$.
	\end{lemma}

	\begin{proof}
		For $f \in A_g$,
		\begin{align*}
			\|f\|_{A_{2g+\log(1+|z|^4)}^2}^2 &= \int_{\C} |f(z)|^2e^{-2g(z)-\log(1+|z|^4)}\,d\la(z)
			\le
			\|f\|_{A_{g}}^2 \int_{\C}\frac{d\la(z)}{1+|z|^4}
		\end{align*}
		implies the first statement.
		For the second claim, we observe that an entire function $f$ fulfills
		$f(z) = \frac{1}{\pi} \int_{|u|\le 1} f(z+u)\,d\la(u)$ for each $z \in \C$,
		which follows from Cauchy's integral formula and switching to polar coordinates.
		Thus,
		\[
		f(z)^2 = \frac{1}{\pi} \int_{|u|\le 1} f(z+u)^2 e^{g(z+u) - g(z+u)}\, d\la(u),
		\]
		and therefore
		\begin{align*}
			|f(z)|^2 \le \frac{1}{\pi} e^{h(z) + K}\int_{\C} |f(z+u)|^2e^{-g(z+u)}\,d\la(u)
		\end{align*}
		which gives the desired result.
	\end{proof}

	\begin{remark}
		The proof shows that $\log(1+|z|^4)$ can be replaced by any function $\rh$ such that $e^{-\rh} \in L^1(\C)$;
		of course, the constant has to be adjusted accordingly.
	\end{remark}

	Let us now show that, under some mild constraints on the family $\cG$, the corresponding inductive limit is regular.

	\begin{proposition}
		\label{prop:regular}
		Let $\cG=(g_k)_k$ be an increasing family of
		continuous functions $g_k : \C \to [0,\infty)$ tending to infinity as $|z|\to \infty$
		such that for all $k$
		\begin{equation} \label{eq:regular}
			\lim_{|z|\rightarrow \infty} g_{k+1}(z)-g_k(z) = \infty.
		\end{equation}
		Then $\cA_{\cG}$ is regular, complete, ultrabornological, reflexive, and webbed.
	\end{proposition}

	\begin{proof}
		We will show that the connecting mappings are compact.
		Then the statements follow from \cite[Satz 25.19, 25.20, 24.23, and Bemerkung 24.36]{MeiseVogt}.

		Take $p:= g_k$ and $q:=g_{k+1}$. We show that the inclusion $A_p \hookrightarrow A_q$ is compact.
		Let $(f_j)$ be a bounded sequence in $A_p$, i.e., there exists $D>0$ such that
		\[
		|f_j(z)|\le D e^{p(z)}, \quad z \in \C,  \text{ for all }j.
		\]
		Then this family of entire functions is locally uniformly bounded.
		Thus, by Montel's theorem, there exists a subsequence $(f_{j_k})$ that converges uniformly on compact subsets to $f \in \cH(\C)$.
		Clearly, $|f(z)|\le D e^{p(z)}$ for all $z\in\C$.
		Let us show that $(f_{j_k})$ converges to $f$ in $A_q$.
		Let $\ve>0$. Choose $R>0$ such that
		\[
		e^{p(z)-q(z)}<\frac{\ve}{2D}, \quad |z|> R,
		\]
		where we use \eqref{eq:regular},
		and let $k_0$ be such that
		\[
		|f_{j_k}(z)-f(z)|<\ve, \quad k \ge k_0,~ |z|\le R.
		\]
		It follows that $\|f_{j_k}-f\|_{A_q}<\ve$
		for $k \ge k_0$.
	\end{proof}

	\subsection{The spaces $\cA_{\Om_\fM^+}$ and $\cA_{\Om_\fM}$}
	\label{sec:Om_A}

	Given a weight matrix $\fM = (M^{(x)})_{x>0}$,
	we consider the sequences of functions
	\begin{align*}
		\Om_\fM^+ &:= \big(z\mapsto k|\Im z|+\om_{M^{(1/k)}}(k z) \big)_k,
		\\
		\Om_\fM &:= \big( z\mapsto \om_{M^{(1/k)}}(kz) \big)_k,
	\end{align*}
	and the associated spaces $\cA_{\Om_\fM^+}$ and $\cA_{\Om_\fM}$.

If the weight matrix is clear from the context, we also write $\om^{(k)}(z):=\om_{M^{(1/k)}}(z)$.
	Note that $\om^{(k)}\le\om^{(l)}$ if $k\le l$ by the definition of associated weight functions. Let us now see that \Cref{prop:regular} is applicable to $\cA_{\Om_\fM^+}$ and $\cA_{\Om_\fM}$.

	Let $l>k>0$. Then \eqref{eq:countingrepr} implies for all $t\ge 0$
$$\om^{(l)}(lt) - \om^{(k)}(kt)\ge\om^{(k)}(lt) - \om^{(k)}(kt)=\int_{kt}^{lt}\frac{\mu_{M^{(1/k)}}(\la)}{\la}\,d\la\ge\mu_{M^{(1/k)}}(kt)\log(l/k).$$
So for all $l>k>0$ we get that $\om^{(l)}(lz) - \om^{(k)}(kz) \to \infty$ as $|z| \to \infty$ and thus we are able to infer the following corollary from Proposition \ref{prop:regular}.

	\begin{corollary}
		\label{cor:dualproperties}
		$\cA_{\Om_\fM^+}$ and $\cA_{\Om_\fM}$ are regular, complete, ultrabornological, reflexive, and webbed.
	\end{corollary}

	In what follows, unless mentioned otherwise, we assume that all weight sequences and matrices are normalized.

	\subsection{The dual of $\cE^{(\fM)}(\R)$}

	Let us recall a result of \cite{Taylor71}. For this we need the Fourier transform of a distribution $T \in \cE(\R)'$,
	\[
	\widehat{T}(z):= T(x\mapsto e^{ixz}).
	\]
	For a weight sequence $M$, set
	\[
	\lambda_M(t):= \sum_{j \ge 0} \frac{t^j}{M_j}, \quad t \ge 0.
	\]
	One immediately infers (cf.\ \cite[Proposition 4.5$(a)\Rightarrow(b)$]{Komatsu73})
	\begin{equation}
		\label{eq:equiv}
		e^{\om_M(t)} \le \lambda_M(t)\le 2e^{\om_M(2t)}, \quad t\ge 0.
	\end{equation}

	\begin{theorem}[{\cite[Theorem 2.8]{Taylor71}}]
		\label{thm:ultrafunctiondual}
		Let $M$ be a weight sequence. Then, for
		\[
		\fE^{(M)}(\R):= \big\{f \in C^\infty(\R): \A K \subset\subset \R, \A m \in \N, \A r >0 : \|f\|^M_{K,m,r}  <\infty \big\},
		\]
		where
		\begin{equation} \label{eq:snormfE}
			\|f\|^M_{K,m,r}:= \sup_{j \in \N,\, 0\le k \le m,\,  x \in K} \frac{|f^{(j+k)}(x)|}{r^j M_j},
		\end{equation}
		endowed with its natural Fr\'echet topology, we have
		\[
		\fE^{(M)}(\R)' \cong \cA_{\Ga_M},
		\]
		where
		\[
		\Ga_{M}:=\big( z \mapsto  k\log(1+|z|) + \log(\lambda_M(k|z|))+k|\Im z| \big)_k,
		\]
		and the isomorphism (of locally convex spaces) is realized by the Fourier transform.
	\end{theorem}

	For a weight matrix $\fM$, we set
	\[
	\fE^{(\fM)}(\R):=\bigcap_{x>0} \fE^{(M^{(x)})}(\R).
	\]
	Note that $\fE^{(M)}(\R) = \cE^{(M)}(\R)$ (resp., $\fE^{(\fM)}(\R)=\cE^{(\fM)}(\R)$),
	if $M$ (resp., $\fM$) is derivation closed.

	\begin{proposition}
		\label{prop:entireiso}
		Let $\fM$ be derivation closed. Then, as locally convex spaces,
		\[
		\cA_{\Ga_\fM} \cong \cA_{\Om_\fM^{+}},
		\]
		where
		\[
		\Ga_{\fM}:= \big( z \mapsto k\log(1+|z|)+\log( \lambda_{M^{(1/k)}}(k|z|)) +k|\Im z| \big)_k .
		\]
	\end{proposition}

	\begin{proof}
		Using \Cref{lem:dcconsequence} together with \eqref{eq:equiv}, one easily checks that the respective inductive systems are equivalent and thus the topologies coincide.
	\end{proof}

	We are ready to identify the dual of $\cE^{(\fM)}(\R)$.

	\begin{theorem}
		\label{thm:ultrafunctiondualmatrix}
		Let $\fM$ be derivation closed. Then
		$\cE^{(\fM)}(\R)' \cong \cA_{\Om_\fM^{+}}$.
	\end{theorem}

	\begin{proof}
		We have $\fE^{(\fM)}(\R)' \cong \cE^{(\fM)}(\R)'$, since $\fM$ is derivation closed.
		First we show
		\[
		\bigcup_{k \in \N_{\ge 1}} \fE^{(M^{(1/k)})}(\R)' \cong \fE^{(\fM)}(\R)',
		\]
		where the union on the left carries the locally convex inductive limit topology
		and the isomorphism is given by the restriction map which we denote by $R$.
		Observe that, for $k \le l$, we have a continuous inclusion $\fE^{(M^{(1/k)})}(\R)' \hookrightarrow \fE^{(M^{(1/l)})}(\R)'$,
		and the locally convex inductive limit exists.

		The map $R$ is surjective, since we can extend each continuous functional on $\fE^{(\fM)}(\R)$ to some $\fE^{(M^{(1/k)})}(\R)$,
		by the Hahn--Banach theorem.
		By  \Cref{lem:density}, this extension is unique and thus $R$ is also injective.

		Let us now show continuity in both directions.
		Continuity of $R$ follows from continuity of its restriction to any fixed $\fE^{(M^{(1/k)})}(\R)'$ which is clear.
		Since $\cE^{(\fM)}(\R)$ is a Fr\'echet--Schwartz space (see the proof of \Cref{lem:toshow}),
		the dual $\cE^{(\fM)}(\R)'$ is bornological (cf.\ \cite[Satz 24.23]{MeiseVogt}).
		Thus it suffices to show that $R^{-1}$ maps bounded sets to bounded sets.
		Now
		\[
		U_n:=\big\{f \in \cE^{(\fM)}(\R): p_n(f):=\|f\|^{M^{(1/n)}}_{[-n,n],1/n} \le \tfrac{1}{n} \big\}, \quad n \in \N_{\ge 1},
		\]
		is a fundamental system of $0$-neighborhoods in $\cE^{(\fM)}(\R)$, and the polars $U_n^\o$ form
		a fundamental system of bounded sets in $\cE^{(\fM)}(\R)'$ (cf.\ \cite[Lemma 25.5]{MeiseVogt}).
		If $T \in U_n^\circ$, then
		\[
		|T(f)|\le np_n(f), \quad f \in \cE^{(\fM)}(\R),
		\]
		and, by the Hahn--Banach theorem, $T$ extends to $\fE^{(M^{(1/n)})}(\R)$ and satisfies this estimate for all $f \in \fE^{(M^{(1/n)})}(\R)$.
		So $R^{-1}(U_n^\circ)$ is bounded in $\fE^{(M^{(1/n)})}(\R)'$, therefore $R^{-1}$ is continuous.

		By \Cref{thm:ultrafunctiondual}, $\fE^{(M^{(1/k)})}(\R)' \cong \cA_{\Ga_{M^{(1/k)}}}$ so that \Cref{prop:entireiso} yields the desired result.
	\end{proof}

	\subsection{The dual of  $\La^{(\fM)}$}

	\begin{theorem}
		\label{lem:sequencedual}
		Let $\fM$ be a weight matrix. Then $(\Lambda^{(\fM)})'\cong A_{\Om_\fM}$,
		and this isomorphism is realized by
		\[
		\widetilde{S}: (\Lambda^{(\fM)})' \rightarrow  A_{\Om_\fM},
		\quad T \mapsto \widetilde{S}(T):=\Big(z \mapsto \sum_{j \ge 0} T(e_j) z^j\Big),
		\]
		with $e_j$ denoting the $j$-th unit vector.
	\end{theorem}

	Since $z \mapsto iz$ is an automorphism of $\C$,
	the map $S(T):= (z \mapsto \sum_{j \ge 0} T(e_j) i^j z^j)$ realizes the isomorphism $(\Lambda^{(\fM)})'\cong A_{\Om_\fM}$, too.

	\begin{proof}
		First observe that for any sequence $b=(b_j)_j$ satisfying
		\begin{equation}
			\label{eq:est1}
			\exists A,B,k>0 \A j\in\N :\quad |b_j| \le A B^j \frac{1}{M_j^{(1/k)}},
		\end{equation}
		the map
		\begin{equation}
			\label{eq:est2}
			T_{b}(a):= \sum_{j \ge 0}a_jb_j,\quad a=(a_j)_j\in\Lambda^{(\fM)},
		\end{equation}
		is an element of $(\Lambda^{(\fM)})'$.
		Actually,
		every $T\in (\Lambda^{(\fM)})'$ has the form \eqref{eq:est2}.
		Indeed,
		\[
		T(a_1 e_1 + \cdots + a_ne_n)=a_1 T(e_1) + \cdots + a_nT(e_n)
		\]
		and since $a_1 e_1 + \cdots + a_ne_n \to a$ in $\La^{(\fM)}$ as $n \to \infty$,
		the statement follows with $b_j = T(e_j)$.

		If $b$ satisfies \eqref{eq:est1}, then
		\[
		f_{b}(z):= \widetilde S(T_b)(z) =  \sum_{j \ge 0}b_jz^j
		\]
		defines an element in $A_{\Om_\fM}$. Indeed,
		\begin{align*}
			|f_{b}(z)| &\le A\sum_{j \ge 0}  \frac{(B |z|)^j}{M_j^{(1/k)}}
			\le A \sup_{k\in\N }\frac{(2B|z|)^k}{M_k^{(1/k)}}\sum_{j \ge 0}\frac{1}{2^j}
			=2Ae^{\om^{(k)}(2Bz)}.
		\end{align*}
		Conversely, if $f \in A_{\Om_\fM}$, then, by the Cauchy estimates and \eqref{assoweight},
		\begin{align*}
			\frac{|f^{(j)}(0)|}{j!}&\le A  \inf_{r > 0} \frac{e^{\om^{(k)}(kr)}}{r^j}
			=Ak^j \frac{1}{M_j^{(1/k)}}.
		\end{align*}
		So $\widetilde{S}: (\Lambda^{(\fM)})' \rightarrow  A_{\Om_\fM}$ is a linear isomorphism.

		Next we show continuity of $\widetilde{S}^{-1}$.
		To this end, it is enough to show that $\widetilde{S}^{-1}\vert_{A_k}$ is continuous,
		where $A_{k}:=A_{\om^{(k)}(kz)}$.
		A typical $0$-neighborhood in $(\Lambda^{(\fM)})'$ is of the form $U = \{T : T(C)\le r \}$
		for some bounded set $C \subseteq \Lambda^{(\fM)}$ and $r>0$.
		Let $D>0$ be such that $|a_j|\le D \frac{1}{(2k)^j}M_j^{(\frac{1}{k})}$ for all $j\in\N$ and all $a=(a_j)_j\in C$.
		Then $\widetilde{S}^{-1}$ maps the $A_k$-ball of radius $\frac{r}{2D}$ into $U$.

		For the continuity of $\widetilde{S}$ we observe that $(\Lambda^{(\fM)})'$ is ultrabornological, since $\Lambda^{(\fM)}$ is a Fr\'echet--Schwartz space (cf.\ \cite[Satz 24.23]{MeiseVogt}).
		Indeed, the Fr\'echet space $\Lambda^{(\fM)}$ is nuclear (by the Grothendieck-Pietsch criterion, cf.\ \cite[28.15]{MeiseVogt})
		and hence a Schwartz space (cf.\ \cite[Corollary 28.5]{MeiseVogt}).
		On the other hand, $A_{\Om_\fM}$ is webbed.
		So the assertion follows from the open mapping theorem (cf.\ \cite[Satz 24.30]{MeiseVogt}).
	\end{proof}

	\section{Proof by dualization}
	\label{sec:dualization}

	This section builds on
	the techniques developed in \cite{BonetMeiseTaylorSurjectivity} for Braun--Meise--Taylor classes.
	Let us first introduce some notation. For a (normalized) non-quasianalytic pre-weight function $\om$, we set
	\begin{align} \label{eq:Pom}
		\begin{split}
			P_\om(x+iy)&:=
			\frac{|y|}{\pi} \int_{-\infty}^\infty \frac{\om(t)}{(t-x)^2+y^2}\,dt,  \quad x,y \in \R,~ y \ne 0
			\\
			P_\om(x)&:= \om(x), \quad x \in \R,
		\end{split}
	\end{align}
	the \emph{harmonic extension} of $\om$ to the open upper and lower half plane (and subharmonic extension to $\C$).
	A detailed exposition of its main features is presented in \Cref{sec:pomega}. $P_\om$ is closely related to
	the concave weight function (cf.\ \cite[Definition 3.1$(b)$]{BonetMeiseTaylorSurjectivity})
	\begin{equation}\label{heir}
		\ka_\om(r):=\int_1^{\infty}\frac{\omega(rt)}{t^2}dt=r\int_r^{\infty}\frac{\omega(t)}{t^2}dt.
	\end{equation}
	In fact,
	\begin{equation}\label{heirswitch}
		\frac{1}{\pi} \ka_\om(r) \le P_{\om}(ir) \le \frac{4}{\pi}\ka_\om(r), \quad r>0,
	\end{equation}
	by \cite[Lemma 3.3]{BonetMeiseTaylorSurjectivity}.
	It was proved in \cite{BonetMeiseTaylorSurjectivity} that, for a non-quasianalytic weight function $\om$ and another weight function $\si$,
	the inclusion
	\[
	\Lambda^{(\si)} \subseteq j^\infty_0\cE^{(\om)}(\R),
	\]
	holds if and only if
	\begin{equation}\label{BMT}
		\ka_\om(r)= O(\si(r))\quad \text{ as }r \rightarrow \infty.
	\end{equation}
	Note that \eqref{BMT} is also equivalent to $\Lambda^{\{\si\}} \subseteq j^\infty_0\cE^{\{\om\}}(\R)$;
	cf.\ \Cref{weighfunction2}, \cite{BonetMeiseTaylorSurjectivity}, and \cite{RoumieuMoment}
	as well as
	\cite{whitneyextensionmixedweightfunction}, \cite{whitneyextensionmixedweightfunctionII},
	and \cite{Rainer:2020ab} for the more general mixed Whitney extension problem.

	When $M$ is a non-quasianalytic weight sequence, we also write $P_M$ instead of $P_{\om_M}$ and $\kappa_M$ instead of $\kappa_{\om_M}$.
	The crucial mixed condition for weight sequences $M, N$ in this section is $M \prec_{L} N$ defined
	by
	\begin{equation}
		\label{eq:L}
		\E C >0 \A s \ge 0:~ P_{N}(is) \le \om_M(Cs)+C.
	\end{equation}
	This condition appears in \cite[(2.14')]{Langenbruch94}.

	Let us now formulate the main result of this section.

	\begin{theorem}
		\label{thm:mainthm2}
		Let $\fM,\fN$ be weight matrices, $\fN$ derivation closed. Then
		\begin{equation} \label{L}
			\tag{L}
			\La^{(\fM)} \subseteq j^\infty_0 \cE^{(\fN)}(\R) \quad \Longleftrightarrow \quad \A y>0 \E x>0 : ~ M^{(x)} \prec_{L} N^{(y)}.
		\end{equation}
	\end{theorem}

	\begin{remark}\label{constantsinout}
		Similarly to \Cref{rem:strongom1},
		note that in \eqref{eq:L} on the right-hand side the constant $C$ appears in the argument of $\om_M$ (not in front).
		This subtle difference will become important later on; it stems from the fact
		that we aim for results for classes defined by (a family of) weight sequences instead of
		(associated) weight functions.
		Even though
		\eqref{heirswitch} implies $\ka_{N}\sim P_N$, generally we cannot replace $P_{N}$ by $\ka_{N}$ in \eqref{eq:L}.
	\end{remark}

	\subsection{Auxiliary results}

	We saw in \Cref{assumptionssuperfluous} that the left-hand side of \eqref{L} entails $\fM (\preceq) \fN$ and
	non-quasianalyticity of $\fN$. This is also true
	for the right-hand side.

	\begin{lemma}\label{weightfctnonquasisuperfluous}
		Let $\fM$ and $\fN$ be weight matrices. The right-hand side of \eqref{L} entails $\fM (\preceq) \fN$ and
		non-quasianalyticity of $\fN$.
	\end{lemma}

	\begin{proof}
		Non-quasianalyticity is immediate, since
		$\om_{N^{(y)}}$ is non-quasianalytic if and only if so is $N^{(y)}$; see \cite[Lemma 4.1]{Komatsu73}.

		Fix $y>0$. There is $x>0$ such that $M^{(x)} \prec_{L} N^{(y)}$, i.e.,
		there is $C>0$ such that
		\[
		\om_{N^{(y)}}(s) \le P_{N^{(y)}}(is) \le  \om_{M^{(x)}}(Cs)+C, \quad s>0,
		\]
		where the first inequality will be justified later in \eqref{eq:est3}.
		By \eqref{assoweight}, we conclude
		\begin{align*}
			N^{(y)}_k &= \sup_{t>0} \frac{t^k}{e^{\om_{N^{(y)}}(t)}}
			\ge \sup_{t>0} \frac{t^k}{e^{\om_{M^{(x)}}(Ct)+C}}= e^{-C} C^{-k}M^{(x)}_k.
		\end{align*}
		This shows $\fM (\preceq) \fN$.
	\end{proof}

	\subsection{Properties of $P_{\om}$, the (sub-)harmonic extension of $\om$ }
	\label{sec:pomega}

	In this section we assume, without further mentioning, that $\om$
	has the following properties:
	\begin{itemize}
		\item $\om:[0,\infty)\rightarrow [0,\infty)$ is increasing and continuous.
		\item $\log(t)=O(\om(t))$ as $t \rightarrow \infty$.
		\item $\ph_\om = \om \o \exp$ is convex.
		\item $\int_0^\infty \frac{\om(t)}{1+t^2}\, dt < \infty$.
	\end{itemize}
	So $\om$ may be any non-quasianalytic pre-weight function, in particular, $\om_M$ for a non-quasianalytic weight sequence $M$.
	Recall our convention $\om(z):= \om(|z|)$ for $z \in \C$.

	The harmonic extension
	$P_{\om}$, defined in \eqref{eq:Pom},
	will play a crucial role as a weight for a weighted space of entire functions.
	Let us list some obvious properties:
	\begin{enumerate}[label=\thetag{\arabic*}]
		\item $P_\om(z)\ge 0$ for all $z \in \C$,
		\item $P_{\om}$ is symmetric relative to the real and imaginary axis,
		\item $\si \le \om$ implies $P_\si \le P_\om$,
		\item $P_{\si+\om}=P_\si+P_\om$,
		\item $P_{t\mapsto \om(nt)}(z)=P_\om(nz)$.
	\end{enumerate}

	\begin{remark}
		\label{rem:generalweight}
		For an increasing sequence of functions $\om_j$ converging to $\om$
		uniformly on compact subsets of $\R$, we get directly from the definition that $P_{\om_j} \rightarrow P_\om$ uniformly on compact subsets of $\C$.
	\end{remark}

	The following proposition is well-known.

	\begin{proposition}
		\label{prop:Psubharmonic}
		$P_{\om}$ is continuous on $\C$, harmonic in the open upper and lower half plane, and subharmonic on $\C$.
	\end{proposition}

	\begin{proof}
		That $P_\om$ is continuous on $\C$ and harmonic in the open upper and lower half plane is clear.
		For the subharmonicity, note first that $\om$ is subharmonic on $\C$; indeed, $\om(z) = \vh_\om(\log|z|)$ and $\vh_\om$ is increasing and convex.

		Next we show that
		\begin{equation}
			\label{eq:est3}
			P_{\om}(z) \ge \om(z),\quad z\in\C.
		\end{equation}
		In fact, $\xi +i\et \mapsto P_{\om}(e^{\xi+i\et})$ is harmonic on the horizontal strip $\{0< \et <\pi\}$, convex in $\xi$, and thus concave and symmetric relative to $\frac{\pi}2$
		in $\et$
		(cf.\ the arguments in \cite[p. 198]{Carleson61}).
		So for any fixed $\xi$ the map $\et \mapsto P_{\om}(e^{\xi+i\et})$ takes its minimum at $\et = 0$ (and $\et=\pi$). Since $P_{\om}$ extends $\om$,
		this proves \eqref{eq:est3}.

		Now, for
		$x \in \R$ and $\de>0$,
		\[
		P_{\om}(x)=\om(x)\le \frac{1}{2\pi} \int_0^{2\pi} \om(x+\de e^{i\th})\,d\th \le \frac{1}{2\pi} \int_0^{2\pi} P_{\om}(x+\de e^{i\th})\,d\th,
		\]
		which implies that $P_\om$ is subharmonic on $\C$.
	\end{proof}

	\subsection{Consequences of properties of weight sequences for $P_{M}$}

	If $M,N$ are weight sequences such that $M \prec_{s\om_1} N$, then one easily infers the existence of $K\ge 1$ such that
	\begin{equation}
		\label{eq:om1consequence}
		\om_M(z+w) \le \om_N(z) + \om_N(w) + K, \quad z,w \in \C.
	\end{equation}
	Moreover, if $M,N$ are non-quasianalytic, there is a constant $C\ge 1$ such that
	\begin{equation}\label{eq:om1consequencePom}
		P_{M}(z)\le P_{N}(z)+C, \quad z \in \C.
	\end{equation}

	\begin{lemma}
		\label{lem:w1}
		Let $M$ and $N$ be weight sequences such that $M$ is non-quasianalytic and $M \prec_{s\om_1} N$.
		Then for all $\varepsilon > 0$ there exists $K>0$ such that
		\[
		P_{M}(x+iy) \le \om_N(x) + \varepsilon y + K, \quad x+iy \in \C.
		\]
	\end{lemma}

	\begin{proof}
		Cf.\  \cite[Lemma 2.2]{BraunMeiseTaylor90} with the obvious changes.
	\end{proof}

	Having this we prove the following mixed version of \cite[Lemma 1.9]{MeiseTaylor88}.

	\begin{lemma}
		\label{lem:mixedw1}
		Let $M^{(i)}$, $1\le i \le 3$, be non-quasianalytic weight sequences with $M^{(1)} \prec_{s\om_1} M^{(2)} \prec_{s\om_1} M^{(3)}$.
		Then there exists $A > 0$ such that
		\[
		P_{M^{(1)}}(z+w) \le P_{M^{(3)}}(z) + A, \quad z,w \in \C,~ |w|\le 1.
		\]
	\end{lemma}

	\begin{proof}
		First observe that $P_{M}$ has the following alternative form
		\begin{equation} \label{eq:alternativeform}
			P_{M}(x+iy)=\frac{1}{\pi}\int_{-\infty}^{\infty} \frac{\om_M(|y|t+x)}{t^2+1}\,dt, \quad (y \ne 0).
		\end{equation}
		Now take $w = u+iv \in \C$ with $|w|\le 1$
		and $z = x+iy \in \C$ with $y>1$. Then $\Im(z+w)= y+v>0$ and, by \eqref{eq:om1consequence},
		\begin{align*}
			P_{M^{(1)}}(z+w) &= \frac{1}{\pi}\int_{-\infty}^{\infty} \frac{\om_{M^{(1)}}((y+v)t + x+u)}{t^2+1}\,dt\\
			&\le \frac{1}{\pi}\int_{-\infty}^{\infty} \frac{\om_{M^{(2)}}(yt+ x) + \om_{M^{(2)}}(vt+ u)+K}{t^2+1}\,dt\\
			&\le P_{M^{(2)}}(z) + K\underbrace{\frac{1}{\pi}\int_{-\infty}^\infty \frac{\om_{M^{(2)}}(|t|+1)+1}{t^2+1}\,dt}_{=:B>1},
		\end{align*}
		since $K\ge 1$ and $\om_{M^{(2)}}(vt+ u)=\om_{M^{(2)}}(|vt+u|)\le\om_{M^{(2)}}(|v||t|+|u|)\le\om_{M^{(2)}}(|t|+1)$.
		By \eqref{eq:om1consequencePom}, the choice $A=BK+C'$ establishes the claim for $y>1$, and by symmetry for $y<-1$.
		If $|y|\le 1$, then \Cref{lem:w1} and \eqref{eq:om1consequence} yield constants $K_i\ge 1$ such that
		\[
		P_{M^{(1)}}(z+w)\le \om_{M^{(2)}}(x+u) + K_1 \le \om_{M^{(3)}}(x) + K_2,
		\]
		and \eqref{eq:est3} finishes the proof.
	\end{proof}

	The effect of derivation closedness on $P_M$ is captured in the next lemma.

	\begin{lemma}\label{dcforPM}
		Let $l \in \N$ and let $M^{(k)}$, for $1\le k \le l+1$, be weight sequences such that $dc(M^{(k)},M^{(k+1)})<\infty$ for all $k$.
		Then there exists $C >0$ such that
		\[
		P_{M^{(l+1)}}(z)+\log(1+|z|^l)\le P_{M^{(1)}}(Cz)+C, \quad z \in \C.
		\]
	\end{lemma}

	\begin{proof}
		By \Cref{lem:dcconsequence}, we have, for some $C>0$,
		\[
		\si(t):= \om_{M^{(l+1)}}(t)+\log(1+t^l)\le \om_{M^{(1)}}(Ct)+C, \quad t \ge 0.
		\]
		It is easy to see that
		$\si$ is a pre-weight function.
		By monotonicity and additivity of $P_\om$ in $\om$, and \eqref{eq:est3} applied to $\log(1+t^l)$, we infer
		\[
		P_{M^{(l+1)}}(z)+\log(1+|z|^l)\le P_\si(z) \le  P_{M^{(1)}}(Cz)+C
		\]
		and are done.
	\end{proof}

	\subsection{Scheme of the proof of \Cref{thm:mainthm2}}

	The key is the following proposition.

	\begin{proposition}[{\cite[Corollary 2.3]{BonetMeiseTaylorSurjectivity}}]
		\label{prop:inclusion}
		Let $E,F, G$ be Fr\'echet--Schwartz spaces and
		let $T \in L(E,F)$ and $R \in L(G,F)$ have dense range.
		Assume that $F'$ endowed with the initial topology with respect to $T^t:F' \rightarrow E'$ is bornological.
		Then the following conditions are equivalent:
		\begin{enumerate}[label=\thetag{\arabic*}]
			\item $R(G) \subseteq T(E)$.
			\item If $B \subseteq F'$ is such that $T^t(B)$ is bounded in $E'$, then $R^t(B)$ is bounded in $G'$.
		\end{enumerate}
		\begin{equation*}
			\begin{tikzcd}
				E	\arrow{r}{T} & F & G \arrow[swap]{l}{R}  && E' &	\arrow[swap]{l}{T^t}  F' \arrow{r}{R^t} & G'
			\end{tikzcd}
		\end{equation*}
	\end{proposition}

	Let $\fM (\preceq )\fN$ be weight matrices, $\fN$ derivation closed and non-quasianalytic.
	We will apply \Cref{prop:inclusion} to
	\begin{equation} \label{eq:diagam0}
		\begin{tikzcd}
			\cE^{(\fN)}(\R)	\arrow{r}{j^\infty_0} & \Lambda^{(\fN)} & \Lambda^{(\fM)} \arrow[swap]{l}{\text{incl}}
		\end{tikzcd}
	\end{equation}
	By \Cref{thm:ultrafunctiondualmatrix} and \Cref{lem:sequencedual},
	we have the following commuting diagram
	\begin{equation}
		\label{diagram}
		\begin{tikzcd}
			\cE^{(\fN)}(\R)'  \arrow[swap]{d}{\cF} & (\Lambda^{(\fN)})'	\arrow[swap]{l}{(j^\infty_0)^t} \arrow{d}{S}
			\arrow{r}{\text{incl}^t}  & (\Lambda^{(\fM)})'  \arrow{d}{S}
			\\
			\cA_{\Om_\fN^+} & \cA_{\Om_\fN} \arrow[swap]{l}{\text{incl}} \arrow{r}{\text{incl}}& A_{\Om_\fM}
		\end{tikzcd}
	\end{equation}
	where the vertical arrows are isomorphisms. This will lead to

	\begin{proposition}
		\label{lem:toshow}
		Let $\fM (\preceq )\fN$ be weight matrices, $\fN$ derivation closed and non-quasianalytic.
		Then the following conditions are equivalent:
		\begin{enumerate}[label=\thetag{\arabic*}]
			\item $\Lambda^{(\fM)} \subseteq j^\infty_0\cE^{(\fN)}(\R)$.
			\item If $B \subseteq \cA_{\Om_\fN}$ is bounded in $\cA_{\Om^+_\fN}$, then $B$ is bounded in $\cA_{\Om_\fM}$.
		\end{enumerate}
	\end{proposition}

	We will prove \Cref{lem:toshow} in \Cref{sec:proofdual}.
	In \Cref{sec:proofmain2} we will make the connection between condition (2) and the right-hand side of \eqref{L}, and
	thus complete the proof of \Cref{thm:mainthm2}.

	\subsection{Proof of \Cref{lem:toshow}} \label{sec:proofdual}

	The proof is based on the following Phragm\'en--Lindel\"of theorem;
	cf.\ \cite[Theorem 6.5.4]{Boas54}.

	\begin{theorem}
		\label{thm:phrag}
		Let $f$ be holomorphic in the upper half plane and continuous up to the boundary. Assume that the zeros of $f$ have no finite limit point, and
		\begin{equation} \label{eq:PLcond}
			\liminf_{r \rightarrow \infty} \frac{\sup_{|z|=r}\log|f(z)|}{r}<\infty, \quad \int_{-\infty}^\infty\frac{\max(0,\log|f(t)|)}{1+t^2}\,dt< \infty.
		\end{equation}
		Then (writing $z=x+iy$)
		\begin{align*}
			\log|f(z)| \le \frac{y}{\pi} \int_{-\infty}^\infty \frac{\log|f(t)|}{(t-x)^2+y^2}\,dt + \frac{2y}{\pi} \lim_{r\rightarrow \infty} \frac{1}{r} \int_0^\pi \log|f(re^{i\th})|\sin(\th)\,d\th.
		\end{align*}
	\end{theorem}

	Every function $f \not \equiv 0$ in $\cA_{\Om_\fN}$, or $\cA_{\Om_\fN^+}$, satisfies the assumptions of \Cref{thm:phrag}. Indeed,
	since $f$ is entire, its zeros cannot have any finite limit point unless $f \equiv 0$.
	Since all $N \in \fN$ are non-quasianalytic and so $\om_N(t) = o(t)$, see \Cref{assofctsection},
	also the conditions \eqref{eq:PLcond} are clear.

	A direct application of this result yields the following corollary.

	\begin{corollary}
		\label{cor:phrag}
		Let $N$ be a non-quasianalytic weight sequence. Let $f \in \cH(\C)$ be such that, for some positive integer $k$,
		\begin{gather*}
			\log|f(z)| = o(|z|) \quad \text{as } |z| \rightarrow \infty
			\quad \text{ and }\quad
			\log|f(x)|	\le  \om_N(kx), ~ x \in \R.
		\end{gather*}
		Then
		\[
		|f(z)| \le e^{ P_{N} (kz)}, \quad  z \in \C.
		\]
	\end{corollary}

	\begin{proof}[Proof of \Cref{lem:toshow}]
		We have to verify the assumptions of \Cref{prop:inclusion} with the choices of \eqref{eq:diagam0}.

		We saw in the proof of \Cref{lem:sequencedual} that $\Lambda^{(\fN)}$ and $\Lambda^{(\fM)}$ are Fr\'echet--Schwartz spaces.
		For $\cE^{(\fN)}(\R)$,
		this is a consequence of the compactness of the inclusions $\cE^{N^{(\frac{1}{k+1})},\frac{1}{k+1}}(K) \hookrightarrow \cE^{N^{(\frac{1}{k})},\frac{1}{k}}(K)$
		for compact intervals $K \subseteq \R$
		(cf.\ \cite[\S 22, Satz 3.1]{FloretWloka}); here $\cE^{M,a}(K)$ denotes the normed space of all functions $f \in C^\infty(K)$ such that $\|f\|^M_{K,a}<\infty$.

		Both maps in \eqref{eq:diagam0} have dense range, since the finite sequences are dense in $\Lambda^{(\fN)}$.

		Next we prove that $(\Lambda^{(\fN)})'$ endowed with the initial topology with respect to $(j^\infty_0)^t : (\Lambda^{(\fN)})' \to \cE^{(\fN)}(\R)'$ is bornological.
		By \eqref{diagram}, this amounts to showing that
		\begin{equation}
			\label{eq:equiv2}
			\cA_{\Om_\fN} \text{ endowed with the trace topology of } \cA_{\Om_\fN^+} \text{ is bornological}.
		\end{equation}
		To prove \eqref{eq:equiv2}, we set
		\[
		\om^{(k)}(z):=\om_{N^{(\frac{1}{k})}}(z),\quad A_k:= A_{k|\Im z|+\om^{(k)}(kz)}.
		\]
		For every $k\in\N_{\ge 1}$, there exists $l>k$ such that $A_k \hookrightarrow A_l$ is compact; cf.\ \Cref{sec:Om_A}.
		Thus, by \cite[Proposition 2.6$(2)\Leftrightarrow(3)$]{BonetMeiseTaylorSurjectivity},
		\eqref{eq:equiv2} holds if and only if
		\begin{equation}
			\label{eq:equiv3}
			\bigcup_{k \in \N_{\ge 1}} \overline{Y}_k^{A_k}  = \overline{\cA_{\Om_\fN}}^{\cA_{\Om_\fN^+}}, \quad \text{ where } Y_k:= \cA_{\Om_\fN} \cap A_k.
		\end{equation}
		The inclusion $\bigcup_{k \in \N_{\ge 1}} \overline{Y}_k^{A_k}  \subseteq \overline{\cA_{\Om_\fN}}^{\cA_{\Om_\fN^+}}$ is clear and we are left to prove the converse.
		To this end, we will show the two inclusions
		\begin{equation}
			\label{eq:closure}
			\overline{\cA_{\Om_\fN}}^{\cA_{\Om_\fN^+}} \subseteq \cA_{\cP_\fN}\subseteq \bigcup_{k \in \N_{\ge 1}}\overline{Y_k}^{A_k},
		\end{equation}
		where we put $P^{(k)}:=P_{\om^{(k)}}$ and $\cP_\fN:= \big(z \mapsto  P^{(k)}(kz) \big)_k$.

		Let us start with the first inclusion in \eqref{eq:closure}.
		By \eqref{eq:est3} and \Cref{lem:w1} (and \Cref{rem:omega1wlog}),
		for each $k \in \N_{\ge 1}$ there exist $l\in \N_{\ge 1}$ and $A>0$ such that
		\begin{equation}
			\label{eq:estimate}
			\om^{(k)}(z) \le P^{(k)}(z) \le \om^{(l)}(z) + |\Im z| + A, \quad z \in \C.
		\end{equation}
		This shows $\cA_{\Om_{\fN}} \subseteq \cA_{\cP_\fN} \subseteq \cA_{\Om_\fN^+}$.
		Thus it suffices to show that $\cA_{\cP_\fN}$ is closed in $\cA_{\Om_\fN^+}$, i.e.,
		$\cA_{\cP_\fN}\cap A_k$ is closed in $A_k$ for all $k$ (cf.\ \cite[\S 25, Satz 1.2]{FloretWloka}).
		So let $f \in \overline{\cA_{\cP_\fN}\cap A_k}^{A_k}$.
		Then there is a sequence $f_j\in \cA_{\cP_\fN}\cap A_k$ converging to $f$ in $A_k$. Clearly, there exists $C>0$ such that
		\[
		|f_j(z)| \le C e^{k|\Im z| +  \om^{(k)}(kz)}, \quad z\in \C, ~j\in \N.
		\]
		Since $f_j \in \cA_{\cP_\fN}$, there exist $C_j>0$ and $k_j$ such that
		\[
		|f_j(z)| \le C_j e^{P^{(k_j)}(k_jz)}, \quad z \in \C,
		\]
		consequently, $\log|f_j(z)|=o(|z|)$ as $|z|\to \infty$.
		Now \Cref{cor:phrag} implies that all $f_j$ are contained and uniformly bounded in some step of $\cA_{\cP_\fN}$.
		This shows that $f\in \cA_{\cP_\fN}$, and we are done.

		It remains to prove
		the second inclusion in \eqref{eq:closure}.
		To this end, we use \cite[Theorem 1]{Taylor71a} which states the following:
		\emph{Let $(\ph_j)_j$ be an increasing sequence of subharmonic functions on $\C$ converging to some subharmonic function $\ph$,
			and assume that $e^{-\ph_1}$ is locally integrable on $\C$.
			Then any function in $A^2_\ph$ can be approximated in $L^2_{\ph(z)+\log(1+|z|^2)}$ by a sequence in $\bigcup_{k \in \N_{\ge 1}} A_{\ph_k(z) + \log(1+|z|^2)}^2$.}

		Let $f \in \cA_{\cP_\fN}$.
		By \Cref{lem:Apincl},
		there exists $k \in \N_{\ge 1}$ such that $f \in A^2_{\vh}$, where $\vh(z) := 2P^{(k)}(kz)+\log(1+|z|^4)$.
		For this $k$, we introduce the function $\om_j$ by
		\begin{align*}
			\om_j(t):= 2\om^{(k)}(kt), ~ |t|\le j,\quad
			\om_j(t):= a_j\log |t| + b_j, ~ |t|\ge j,
		\end{align*}
		where $a_j, b_j\in\R$ are chosen such that $\om_j$ is continuous, increasing, and $t\mapsto \om_j(e^t)$ is convex.
		Then $\big(\vh_j(z):= P_{\om_j}(z)+\log(1+|z|^4)\big)_j$ is an increasing sequence of subharmonic functions converging to $\vh$;
		cf.\ \Cref{rem:generalweight}. Thus there is a sequence $(f_j)_j$ such that
		$f_j \in A^2_{\vh_j(z) +  \log(1+|z|^2)}$ and
		$f_j \to f$ in $A^2_{\vh(z) +\log(1+|z|^2)}$.
		By \eqref{eq:estimate} and \Cref{lem:dcconsequence},
		there exist $s \in \N_{\ge 1}$ and $K\ge 1$ such that
		\begin{align*}
			\vh(z) +\log(1+|z|^2)
			&\le
			2\om^{(l)}(kz)+ 2k |\Im z| +\log(1+|z|^4) + \log(1+|z|^2) +2A
			\\
			&\le 2 \om^{(s)}(sz)+ 2s |\Im z|+K
		\end{align*}
		for all $z\in\C$ so that $A^2_{\vh(z) +\log(1+|z|^2)} \hookrightarrow A^2_{2 \om^{(s)}(sz)+ 2s |\Im z|}$.
		By \Cref{lem:mixedw1} and \eqref{eq:estimate}, there exist $t \in \N_{\ge 1}$ and $L\ge 1$ such that
		\[
		2 \om^{(s)}(s(z+u))+2 s|\Im (z+u)| \le 2\om^{(t)}(tz)+2t|\Im z|+L, \quad z,u \in \C,~ |u|\le 1.
		\]
		Then \Cref{lem:Apincl} implies $A^2_{\vh(z) +\log(1+|z|^2)} \hookrightarrow  A_{\om^{(t)}(tz)+t|\Im z|} = A_t$.
		Thus $f_j \to f$ in $A_t$.
		Since $P_{\om_j}(z)=O(\log |z| )$ as $|z| \rightarrow \infty$, all $f_j$ are actually polynomials and hence contained in $Y_t$.
		So also the second inclusion in \eqref{eq:closure} is proved.
	\end{proof}

	\subsection{Proof of \Cref{thm:mainthm2}} \label{sec:proofmain2}

	Let $\fM (\preceq) \fN$ be weight matrices, $\fN$ derivation closed and
	non-quasianalytic; cf.\ \Cref{weightfctnonquasisuperfluous}.
	We write $\om^{(k)}(z):= \om_{N^{(\frac{1}{k})}}(z)$.

	We need the following lemma.

	\begin{lemma}  \label{lem:extracted}
		Let $a_j\ge 1$ be a sequence tending to $\infty$ and $k_0$ a positive integer.
		There exist a sequence of polynomials $(p_j)_j$ and $k\in\N_{\ge k_0}$ such that
		$p_j(ia_j) = e^{P^{(k_0)}(ia_j)}$ and
		\begin{align}
			\label{eq:bounded}
			|p_j(z)| \le Ce^{P^{(k)}(Dz)}, \quad z \in \C,~ j\ge 1,
		\end{align}
		with uniform constants $C,D>0$.
	\end{lemma}

	\begin{proof}
		We follow closely the arguments in the proof of \cite[Proposition 2.3]{MeiseTaylor88}.
		Let $k_1\le k_2$ be chosen such that $N^{(\frac{1}{k_0})} \prec_{s\om_1} N^{(\frac{1}{k_{1}})} \prec_{s\om_1} N^{(\frac{1}{k_{2}})}$;
		cf.\ \Cref{rem:omega1wlog}.
		We can find positive numbers $A_j$, $B_j$, and $R_j$ such that
		\[
		\om_j(t):= \begin{cases}
			\om^{(k_2)}(t),  & |t| \le R_j, \\
			A_j\log |t|  + B_j, & |t| >R_j,
		\end{cases}
		\]
		is continuous, increasing, $t \mapsto \om_j(e^t)$ is convex, $\om_j\le \om^{(k_2)}$,
		and
		\begin{equation} \label{eq:approxP}
			\sup_{|z-ia_j|\le 1} |P^{(k_2)}(z) - P_{\om_j}(z)| \le \frac{1}{j}, \quad \text{ for all } j;
		\end{equation}
		cf.\ \Cref{rem:generalweight}.
		Let $\ph : \C \to  [0,1]$ be a $C^\infty$-function with support contained in the unit disc and $\ph(z)=1$ for $|z|\le \frac{1}{2}$.
		As in \cite{MeiseTaylor88}, we set
		\[
		u_j(z):= \Big(1-\frac{z}{ia_j}\Big)^{-1} e^{P^{(k_0)}(ia_j)} \,\db \ph(z-ia_j).
		\]
		By \Cref{lem:mixedw1}, there is $A>0$ such that, for all $j$,
		\begin{equation}
			\label{eq:k0k2}
			P^{(k_0)}(ia_j) \le P^{(k_2)}(z)+A, \quad |z - i a_j|\le 1.
		\end{equation}
		Thus, there exists $M\ge 1$ such that for all $j$ we have
		\begin{align*}
			\MoveEqLeft \int_{\C} |u_j(z)|^2 e^{-2P_{\om_j}(z) - \log(1+|z|^2)}d\la (z)
			\\  &= \int_{|z-ia_j|\le 1} |u_j(z)|^2 e^{-2 P_{\om_j}(z)- \log(1+|z|^2)}d\la (z) \le M.
		\end{align*}
		Since $\db u_j= 0$,
		we infer from \cite[Theorem 4.4.2]{hoermandercomplex} the existence of
		$v_j \in C^\infty(\C)$ with $\db v_j=u_j$ such that
		\[
		\int_{\C} |v_j(z)|^2 e^{-2 P_{\om_j}(z) - 3\log(1+|z|^2)}\,d\la(z)\le M.
		\]
		Then
		\[
		p_j(z):= \ph(z-ia_j)e^{P^{(k_0)}(ia_j)} - \Big(1-\frac{z}{ia_j}\Big)v_j(z)
		\]
		is entire and $p_j(ia_j) = e^{P^{(k_0)}(ia_j)}$.

		We claim that there exists $M'>0$ such that, for all $j$,
		\begin{equation}
			\label{eq:pjest}
			\int_\C |p_j(z)|^2e^{-2 P_{\om_j}(z) - 4\log(1+|z|^2)} d\la(z) \le  M'.
		\end{equation}
		Indeed, by \eqref{eq:k0k2} and \eqref{eq:approxP},
		\begin{align*}
			\MoveEqLeft \int_\C |\ph(z-ia_j)|^2e^{2P^{(k_0)}(ia_j)}e^{-2 P_{\om_j}(z) - 4\log(1+|z|^2)} d\la(z)\\
			\le & e^{2A}\int_{|z-ia_j|\le 1} e^{2(P^{(k_2)}(z) - P_{\om_j}(z)) -  4\log(1+|z|^2)} d\la(z)\\
			\le & e^{2A}\int_{|z-ia_j|\le 1} e^{\frac{2}{j} -  4\log(1+|z|^2)} d\la(z),
		\end{align*}
		which is bounded in $j$.
		And, since $|1-\frac{z}{ia_j}|^2 \le 2(1+|z|^2)$,
		\begin{align*}
			\MoveEqLeft \int_\C \Big|1-\frac{z}{ia_j}\Big|^2|v_j(z)|^2e^{-2 P_{\om_j}(z) - 4\log(1+|z|^2)} d\la(z)\\
			\le &~2\int_\C |v_j(z)|^2e^{-2P_{\om_j}(z) - 3\log(1+|z|^2)} d\la(z) \le 2M.
		\end{align*}
		This yields \eqref{eq:pjest}.
		Since $2 P_{\om_j}+4\log(1+|z|^2)= O(\log(1+|z|^2))$ as $|z| \to \infty$,
		we infer that $p_j$ is actually a polynomial.

		Let us show \eqref{eq:bounded}. Recall that $P_{\om_j} \le P^{(k_2)}$.
		By \Cref{dcforPM}, we find $k_3 \in \N_{\ge 1}$ and $K>0$ such that
		\[
		P^{(k_2)}(z)+\log(1+|z|^2)\le P^{(k_3)}(k_3z)+K, \quad z \in \C.
		\]
		Together with \eqref{eq:pjest} this yields that
		$(p_j)_j$ is bounded in $A^2_{2 P^{(k_3)}(k_3z)}$.
		Take integers $k_4\le k_5$
		such that
		$N^{(\frac{1}{k_3})} \prec_{s\om_1} N^{(\frac{1}{k_{4}})} \prec_{s\om_1} N^{(\frac{1}{k_{5}})}$.
		By \Cref{lem:mixedw1}, there is $K_1>0$ such that
		\[
		2 P^{(k_3)}(k_3(z+w)) \le 2 P^{(k_5)}(k_5z) +K_1, \quad z,w \in \C,~ |w|\le 1.
		\]
		Combining this with \Cref{lem:Apincl}, we find that $(p_j)_j$ is bounded in $A_{P^{(k_5)}(k_5z)}$,
		which shows \eqref{eq:bounded} and thus finishes the proof.
	\end{proof}

	\begin{proof}[Proof of \Cref{thm:mainthm2}]
		By \Cref{lem:toshow}, we need to show that the following conditions are equivalent:
		\begin{enumerate}[label=\thetag{\arabic*}]
			\item $\forall y>0 \E x>0:~ M^{(x)} \prec_{L} N^{(y)}$.
			\item If $B \subseteq \cA_{\Om_\fN}$ is bounded in $\cA_{\Om^+_\fN}$, then $B$ is bounded in $\cA_{\Om_\fM}$.
		\end{enumerate}

		(1) $\Rightarrow$ (2)
		Let $B\subseteq \cA_{\Om_\fN}$ be bounded in $\cA_{\Om_\fN}^+$.
		So there exist $C >0$ and $k \in \N_{\ge 1}$ such that
		\begin{equation*}
			|f(z)| \le C e^{k|\Im z| + \om^{(k)}(kz)}, \quad z \in \C,~ f \in B,
		\end{equation*}
		where we again use the notation $\om^{(k)}(z):= \om_{N^{(\frac{1}{k})}}(z)$.
		Since $B \subseteq \cA_{\Om_\fN}$, we have
		\[
		|f(z)| \le C_f e^{\om^{(k_f)}(k_fz)}, \quad z \in \C,
		\]
		which yields $\log|f(z)|=o(|z|)$ as $|z| \to \infty$. Then \Cref{cor:phrag} implies
		\[
		|f(z)| \le Ce^{P^{(k)}(kz)}, \quad z \in \C,~ f \in B,
		\]
		where $P^{(k)}:=P_{\om^{(k)}}$. By (1) (and the arguments after \eqref{eq:est3}),
		there exist $l \in \N_{\ge 1}$ and $K>0$ such that
		\[
		P^{(k)}(kz)\le P^{(k)}(ik|z|) \le \om_{M^{(\frac{1}{l})}}(lz)+K, \quad z \in \C.
		\]
		This shows that $B$ is bounded in $\cA_{\Om_\fM}$.

		(2) $\Rightarrow$ (1)
		We argue by contradiction.
		Suppose that there exist $k_0\in \N_{\ge 1}$ and a sequence of real numbers $a_j\ge 1$ tending to infinity
		such that for all $j$
		\begin{equation} \label{eq:2to1contra}
			P^{(k_0)}(i a_j) \ge \om_{M^{(\frac1j)}}(j a_j) + j.
		\end{equation}
		By \Cref{lem:extracted}, there is a sequence of polynomials $(p_j)_j$ and $k\in\N_{\ge k_0}$ such that
		$p_j(ia_j) = e^{P^{(k_0)}(ia_j)}$.
		This gives the desired contradiction:
		The sequence $(p_j)_j$ is contained in
		$\cA_{\Om_\fN}$, since
		$\log(|z|)=o(\om^{(k)}(z))$ as $|z| \to \infty$.
		By \eqref{eq:bounded} and \Cref{lem:w1} (in view of \Cref{rem:omega1wlog}),
		$(p_j)_j$ is bounded in $\cA_{\Om_\fN^+}$.
		But, by \eqref{eq:2to1contra}, for every fixed $l\in\N_{\ge 1}$ and $j\ge l$, we have
		\[
		p_j(i a_j) =\exp\big(P^{(k_0)}(ia_j)\big)\ge  e^{j}\exp\big(\om_{M^{(\frac{1}{j})}} (ja_j)\big)\ge e^{j}\exp\big(\om_{M^{(\frac{1}{l})}} (la_j)\big).
		\]
		Thus $(p_j)_j$ is unbounded in every step of the inductive limit defining $\cA_{\Om_\fM}$ and hence in $\cA_{\Om_\fM}$,
		the limit being regular due to Corollary \ref{cor:dualproperties}.
	\end{proof}

	\subsection{Theorem \ref{thm:mainthm2} without derivation closedness}
	If we do not require derivation closedness in Theorem \ref{thm:mainthm2} for $\fN$,
	we still can infer some information on the image of the Borel map, but for a (in general) smaller class.
	Let us be more precise.
	For a weight matrix $\fN=(N^{(\frac{1}{k})})_{k \in \N_{\ge1}}$,
	we may consider the matrix $\fN_{(dc)}=(N_{(dc)}^{(\frac{1}{k})})_{k \in \N_{\ge1}}$
	consisting of ``shifted'' sequences:
	\[
	(N_{(dc)}^{(\frac{1}{k})})_j:=N_{j-k}^{(\frac{1}{k})} \text{ for } j \ge k, \quad (N_{(dc)}^{(\frac{1}{k})})_j := 1 \text{ for } j <k.
	\]
	Then $\fN_{(dc)}$ is easily seen to be derivation closed, and $N_{(dc)}^{(\frac{1}{k})} \le N^{(\frac{1}{k})}$ for all $k$.
	(For a single weight sequence $N$, we may still perform this construction with $N^{(\frac1k)}:= N$ for all $k$
	which leads to a derivation closed matrix $\fN_{(dc)}$ such that $\cE^{(\fN_{(dc)})} \subseteq \cE^{(N)}$.)

	We get the following version of Theorem \ref{thm:mainthm2}.
	\begin{theorem}
		\label{thm:mainthm2dc}
		Let $\fM,\fN$ be weight matrices. Then
		\begin{align*}
			&\La^{(\fM)} \subseteq j^\infty_0 \cE^{(\fN_{(dc)})}(\R) \\
			&\Longleftrightarrow ~ \A y>0 \A n \in \N \E x,C>0 \A t \ge 0: ~ P_{N^{(y)}}(it) + \log(1+t^n) \le \om_{M^{(x)}}(Ct)+C.
		\end{align*}

	\end{theorem}

	\begin{proof}
		Observe that
		\[
		\om_{N^{(\frac{1}{n})}}(t)+\log(1+t^{n})-C \le \om_{N_{(dc)}^{(\frac{1}{n})}}(t)\le \om_{N^{(\frac{1}{n})}}(t)+\log(1+t^{n})+C.
		\]
		These inequalities transfer also to the respective harmonic extensions. And this immediately yields the result via an application of Theorem \ref{thm:mainthm2}.
	\end{proof}

	\section{Comparison and conclusions}
	\label{sec:comparison}

	Let us apply our results to Denjoy--Carleman and Braun--Meise--Taylor classes
	and compare them with the known classical extension results.

	\subsection{Denjoy--Carleman classes}

	Taking $\fM = (M)$ and $\fN = (N)$ in \Cref{thm:mainthm1},
	we recover the Beurling result of \cite{surjectivity} (see also \cite{mixedramisurj}).

	\begin{remark}
		It might be irritating that $\fN = (N)$ clearly fails \eqref{liminfimportant}, which was used in the proof of \Cref{thm:mainthm1} in a crucial way,
		but \Cref{lem:ML} associates with $N$ an equivalent weight matrix with the desired properties (which consists of infinitely many different weight sequences
		that are however all equivalent to $N$).
	\end{remark}

	Let us now discuss \Cref{thm:mainthm2} in this special setting.
	Let $M, N$ be weight sequences, and $N$ in addition derivation closed, such that $(m_k)^{1/k}$ and $(n_k)^{1/k}$ tend to $\infty$.
	Let $\fM=(M^{(x)})_{x>0}$, $\fN=(N^{(y)})_{y>0}$ be the weight matrices equivalent to $M$, $N$, respectively, provided by \Cref{lem:ML}.
	Thus,
	for each $k\in\N_{\ge 1}$ there are constants $A_k,B_k>0$ such that,
	for all $t\ge 0$,
	\begin{gather*}
		\label{eq:equiva}
		\om_{M}(2^kt) -\log(B_k)\le \om_{M^{(\frac{1}{k})}}(t) \le \om_M(2^kt) -\log(A_k),
		\\
		\om_{N}(2^kt) -\log(B_k)\le \om_{N^{(\frac{1}{k})}}(t) \le \om_N(2^kt) -\log(A_k),
	\end{gather*}
	and consequently,
	\begin{equation*}
		\label{eq:equival}
		P_{N}(2^kt) -\log(B_k)\le P_{N^{(\frac{1}{k})}}(t) \le  P_N(2^kt) -\log(A_k).
	\end{equation*}
	This now shows that the right-hand side of \eqref{L} reduces to
	\[
	\E C>0 \A t\ge0:~ P_N(it) \le \om_M(Ct)+C,
	\]
	i.e., $M \prec_{L} N$.

	In this case, \Cref{thm:mainthm2} specializes to a version of \cite[Theorem 2.3]{Langenbruch94} (see also the remarks
	before Corollary 2.4 in said paper).
	Incorporating the Roumieu case (see \cite{surjectivity} and \cite{mixedramisurj}) and
	the implications of \Cref{thm:mainthm1}, we conclude

	\begin{theorem}\label{sequencesrleations}
		Let $M,N$ weight sequences, $N$ derivation closed, with $(m_k)^{1/k} \to \infty$ and $(n_k)^{1/k} \to \infty$. Then the following are equivalent:
		\begin{enumerate}[label=\thetag{\arabic*}]
			\item $\La^{(M)} \subseteq j^\infty_0\cE^{(N)}(\R)$.
			\item $\La^{\{M\}} \subseteq j^\infty_0 \cE^{\{N\}}(\R)$.
			\item $M \prec_{L} N$.
			\item $M \prec_{SV} N$.
		\end{enumerate}
		If $M$ has moderate growth, then the conditions are also equivalent to
		\begin{enumerate}[label=\thetag{5}]
			\item There is $C>0$ such that $\kappa_{N}(s)=O(\om_M(Cs))$ as $s \to \infty$.
		\end{enumerate}
	\end{theorem}

	In fact, that (3) implies (5) follows from \eqref{heirswitch}.
	And, for (5) $\Rightarrow$ (3)
	note that moderate growth of $M$ is equivalent to
	\[
	\exists H\ge 1 \A t\ge 0:\quad 2\omega_M(t)\le\omega_M(Ht)+H,
	\]
	see \cite[Proposition 3.6]{Komatsu73}, which allows to ``move constant factors in front of $\omega_M$ to its argument''.

	Finally, we want to make the connection to the condition $M \prec_{\ga_1} N$ defined by
	\[
	\sup_{j\ge 1}\frac{\mu_j}{j}\sum_{k\ge j}\frac{1}{\nu_k}<+\infty.
	\]
	Note that $M \prec_{\ga_1} M$ is the condition $(\gamma_1)$ in \cite{petzsche} and $(M.3)$ in \cite{Komatsu73}.
	If $M$ is a weight sequence, then $M \prec_{\ga_1} M$ and $M \prec_{SV} M$ are equivalent (see \cite[Theorem 3.6]{surjectivity} and \cite[Theorem 5.2]{mixedramisurj}),
	but in the mixed setting they fall apart, in general.
	For weight sequences $M,N$ such that $M \le C N$ for some $C\ge 1$ we have that
	$M \prec_{\ga_1} N$ implies $M \prec_{SV} N$.
	If additionally $M$ has moderate growth, then $M \prec_{\ga_1} N$ if and only if $M \prec_{SV} N$ since these conditions persist if $M$ (or $N$) is replaced by an equivalent weight sequence and since $M$ has moderate growth if and only if $\mu_k \le C_1 (M_k)^{1/k}$ (see e.g. \cite[Lemma 2.2]{whitneyextensionweightmatrix}).
	Thus, under this additional requirement on $M$, if $M_k \le C N_k$, then also $\mu_k \le C_2 \nu_k$. Invoking \cite[Lemma 5.7]{whitneyextensionmixedweightfunction}, we see that, under these circumstances, $M \prec_{\ga_1} N$ implies
	$\kappa_{N}(s)=O(\om_M(s))$ as $s \to \infty$ as well.

	So we have the following supplement.

	\begin{supplement} \label{supplement}
		In the setting of \Cref{sequencesrleations},
		if $M$ has moderate growth and $M \le C N$, then the conditions (1)--(5) are further equivalent to each of the following conditions:
		\begin{enumerate}
			\item[\thetag{6}] $M \prec_{\ga_1} N$.
			\item[\thetag{7}]  $\kappa_{N}(s)=O(\om_M(s))$ as $s \to \infty$.
		\end{enumerate}
	\end{supplement}

	Clearly, (5) and (7) are equivalent if $\om_M$ is a weight function, but in general it is just a pre-weight function.

	\subsection{Braun--Meise--Taylor classes}

	Let $\Sigma=(S^{(x)})_{x>0}$ and $\Om=(W^{(x)})_{x>0}$ be the matrices associated with the weight functions $\sigma$ and $\om$, respectively.
	By \Cref{Ompropcollect}, the basic assumptions in \Cref{thm:mainthm2} hold for the choices $\fM=\Sigma$, $\fN=\Om$, provided that $\om(t)=o(t)$ and $\sigma(t)=o(t)$ as $t \to \infty$.
	By \eqref{omOMlocallyequal},
	\[
	\La^{(\sigma)}\cong \La^{(\Sigma)},\quad \cE^{(\Om)}(\R) \cong \cE^{(\om)}(\R).
	\]
	In this case, the right-hand side of \eqref{L}, i.e., for all $y>0$ there is $x>0$ with $S^{(x)} \prec_{L} W^{(y)}$ which amounts to
	\begin{equation} \label{Lom}
		P_{W^{(y)}}(is) \le \om_{S^{(x)}}(Cs)+C, \quad s \ge 0,
	\end{equation}
	is equivalent to \eqref{BMT}, i.e.,  $\ka_\om(t) = O(\si(t))$ as $t \to \infty$. Indeed,
	by \Cref{Ompropcollect}, we have $\omega\sim\omega_{W^{(x)}},\sigma\sim\omega_{S^{(x)}}$ and thus, by definition, $\kappa_{\om}\sim\kappa_{W^{(x)}}$ and $P_{\om}\sim P_{W^{(x)}}$ for all $x>0$,
	whence one implication follows from \eqref{heirswitch}.
	Conversely, let $y>0$ be given. Then, for all $x>0$,
	\begin{align*}
		P_{W^{(y)}}(is) \stackrel{\eqref{heirswitch}}{\le}  \frac{4}{\pi}\ka_{W^{(y)}}(s)
		\stackrel{\eqref{goodequivalenceclassic}}{\le} \frac{4}{y\pi}\ka_{\om}(s)
		&\le \frac{4C}{y\pi}(\sigma(s)+1)
		\\
		&\stackrel{\eqref{goodequivalenceclassic}}{\le}
		\frac{4C}{y\pi}(2x \om_{S^{(x)}}(s)+  D_x+1),
	\end{align*}
	and \eqref{Lom} follows if we put $x := \frac{y\pi}{8C}$.

	In this case, \Cref{thm:mainthm2} specializes to \cite[Theorem 3.6]{BonetMeiseTaylorSurjectivity}.
	Incorporating also the Roumieu part \cite[Theorem 3.7]{BonetMeiseTaylorSurjectivity} (see also \cite[Section 5]{RoumieuMoment}) and
	the implications of \Cref{thm:mainthm1}, we find

	\begin{theorem}\label{weighfunction2}
		Let $\omega, \sigma$ be weight functions satisfying $\om(t)=o(t)$, $\sigma(t)=o(t)$ as $t \to \infty$
		and let $\Om=(W^{(x)})_{x>0}$, $\Sigma=(S^{(x)})_{x>0}$ be the associated weight matrices.
		Then the following conditions are equivalent:
		\begin{enumerate}[label=\thetag{\arabic*}]
			\item $\La^{(\si)}\subseteq j^\infty_0\cE^{(\om)}(\R)$.
			\item $\La^{\{\si\}}\subseteq j^\infty_0\cE^{\{\om\}}(\R)$.
			\item $\ka_\om(t) = O(\si(t))$ as $t \to \infty$.
			\item For all $y>0$ there is $x>0$ such that $S^{(x)} \prec_{SV} W^{(y)}$.
			\item For all $y>0$ there is $x>0$ such that $S^{(x)} \prec_{L} W^{(y)}$.
			\item There are $x,y>0$ such that $\ka_{W^{(y)}}(t) = O(\om_{S^{(x)}}(t))$ as $t \to \infty$.
		\end{enumerate}
	\end{theorem}

	Note that any of the six conditions implies that  $\om$ is  non-quasianalytic;
	cf.\  \Cref{weightfctnonquasisuperfluous}.

	\appendix
	\section{Density of entire functions}
	\label{sec:density}

	The following lemma is probably well-known, but we include a proof for the convenience of the reader. The proof closely follows the arguments in
	\cite[Proposition 3.2]{Hoermander85} and \cite[Proposition 3.2]{HeinrichMeise07}.

	\begin{lemma} \label{lem:density}
		Let $M$ be a weight sequence with $(m_j)^{1/j}\rightarrow \infty$.
		Let $f \in \fE^{(M)}(\R)$,
		and $I_k := [-k,k]$.
		Then there exists a sequence of entire functions $f_j$ such that
		$\|f-f_j\|^M_{I_k,m,r} \to 0$ for all $m \in \N$ and $r>0$; see \eqref{eq:snormfE} for the definition of the seminorm.
	\end{lemma}

	\begin{proof}
		Let $\ch \in C^\infty(\R)$ be a function with compact support and $0 \le \ch \le 1$.
		Suppose that $\ch$ is $1$ on $[-k-1,k+1]$ and $0$ outside $[-k-2,k+2]$.
		Set $E_j(z):=\sqrt{\frac{j}{\pi}}e^{-jz^2}$ and
		$f_j:=E_j\ast \chi f$. Then $f_j$ is entire.
		It is easily seen by induction on $p$ that
		\[
		f_j^{(p)}(x)=E_j\ast \chi f^{(p)}(x)+ \sum_{\nu = 1}^p E_j^{(p-\nu)} \ast \chi' f^{(\nu-1)}(x).
		\]
		This yields
		\begin{equation*}
			\label{eq:app}
			|(f-f_j)^{(p)}(x)| \le |f^{(p)}(x) - E_j\ast \chi f^{(p)}(x)|+\sum_{\nu = 1}^p |E_j^{(p-\nu)} \ast \chi' f^{(\nu-1)}(x)|.
		\end{equation*}
		For $|x|\le k$, we find (see \cite[(3.5)]{HeinrichMeise07})
		\begin{align*}
			\label{eq:app1}
			|f^{(p)}&(x) - E_j\ast \chi f^{(p)}(x)|
			\le \frac{C_0}{\sqrt{j}}\Big( \sup_{|\xi|\le k+1} |f^{(p+1)}(\xi)| + 2\sup_{|\xi|\le k+2}|f^{(p)}(\xi)| \Big),
		\end{align*}
		for some absolute constant $C_0$.
		Moreover, (see \cite[(3.6) and (3.7)]{HeinrichMeise07})
		\begin{align*}
			|E_j^{(p-\nu)} \ast \chi' f^{(\nu-1)}(x)| &\le D \sup_{|\xi|\le k+2}|f^{(\nu-1)}(\xi)|\sup_{|y| \ge 1}|E_j^{(p-\nu)}(y)|\\
			&\le D \sup_{|\xi|\le k+2}|f^{(\nu-1)}(\xi)|(C_1(p-\nu))^{p-\nu}\sqrt{\frac{j}{\pi}} e^{-C_2 j}
		\end{align*}
		for some constant $D$ depending on $\chi$ and absolute constants $C_1,C_2>0$.
		Since $(m_p)^{1/p} \rightarrow \infty$,
		there is a constant $C_3$ such that $(C_1(p+l))^{p+l}\le C_3  (\tfrac{r}{2})^pM_p$ for all $p\in \N$ and $0\le l \le m$.
		Altogether, we find,
		for $|x|\le k$, $p \in \N$, and  $0\le l \le m$,
		\begin{align*}
			&|(f-f_j)^{(p+l)}(x)|\le \frac{3C_0}{\sqrt{j}} \|f\|^M_{I_{k+2},m+1,r} r^pM_p
			\\
			&\quad+ D \sqrt{\frac{j}{\pi}} e^{-C_2 j} \|f\|^M_{I_{k+2},m,r/2} \Big(C_3\sum_{\nu = 1}^{p}  (\tfrac{r}{2})^{p -1} M_{\nu-1} M_{p-\nu}
			+ C_4 \sum_{\nu = p+1}^{p+l}   (\tfrac{r}{2})^p M_p \Big)
			\\
			&\quad\le \|f\|^M_{I_{k+2},m+1,r} r^pM_p \Big( \frac{3C_0}{\sqrt{j}}  + DC_5 \sqrt{\frac{j}{\pi}} e^{-C_2 j}\Big)
		\end{align*}
		for absolute constants $C_4,C_5$.
		This implies $\|f-f_j\|^M_{I_k,m,r} \to 0$ as $j \to \infty$.
	\end{proof}

	%\bibliographystyle{amsplain}
	%\bibliography{../../../references/biblio, Bibliography}

	\def\cprime{$'$}
	\providecommand{\bysame}{\leavevmode\hbox to3em{\hrulefill}\thinspace}
	\providecommand{\MR}{\relax\ifhmode\unskip\space\fi MR }
	\providecommand{\MRhref}[2]{%
		\href{http://www.ams.org/mathscinet-getitem?mr=#1}{#2}
	}
	\providecommand{\href}[2]{#2}

\end{document}